\documentclass[a4paper,11pt]{article}

\pdfoutput=1 

\usepackage{scalefnt} 

\usepackage[utf8]{inputenc}
\usepackage[T1]{fontenc}
\usepackage[english]{babel}
\usepackage{amsmath} 
\usepackage{amssymb} 
\usepackage{amsthm}  
\usepackage{hyperref}  
\usepackage{tikz}      
\usepackage{tcolorbox} 
\usepackage{dcolumn}   
\usepackage{graphicx, wrapfig, subcaption, setspace, booktabs, float, epsfig} 
\usepackage{multirow}  
\usepackage[all]{xy}

\usepackage{enumitem}
\usepackage{tikz}

\theoremstyle{plain}
\newtheorem{theorem}{Theorem}[section]
\newtheorem{remark}{Remark}
\newtheorem{proposition}{Proposition}
\newtheorem{lemma}{Lemma}[section]

\newtheorem*{problem}{Problem}
\newtheorem{corollary}{Corollary}

\hypersetup{
colorlinks=true,
linkcolor=blue,
citecolor=blue,
filecolor=blue,
pagecolor=blue,
urlcolor=black,
pdfauthor={Yuri Ximenes Martins},
pdftitle={Existence and Classification of Pseudo-Asymptotic Solutions for Tolman--Oppenheimer--Volkoff Systems}}

\begin{document}

\title{\textbf{Existence and Classification of Pseudo-Asymptotic Solutions for Tolman--Oppenheimer--Volkoff Systems}}

\author{Yuri Ximenes Martins\footnote{yurixm@ufmg.br (corresponding author)}, Luiz Felipe Andrade Campos\footnote{luizfelipeac@ufmg.br}, \\ Daniel de Souza Pl\'{a}cido\footnote{danielspteixeira@ufmg.br},  Rodney Josu\'e Biezuner\footnote{rodneyjb@ufmg.br}\\ \\ \textit{Departamento de Matem\'atica, ICEx, Universidade Federal de Minas Gerais,}  \\  \textit{Av. Ant\^onio Carlos 6627, Pampulha, CP 702, CEP 31270-901, Belo Horizonte, MG, Brazil}}


\maketitle

\begin{abstract}
The Tolman--Oppenheimer--Volkoff (TOV) equations are a partially uncoupled system of nonlinear and non-autonomous ordinary differential equations which describe the structure of isotropic spherically symmetric static fluids. Nonlinearity makes finding explicit solutions of TOV systems very difficult and such solutions and very rare. In this paper we introduce the notion of pseudo-asymptotic TOV systems and we show that the space of such systems is at least fifteen-dimensional. We also show that if the system is defined in a suitable domain (meaning the extended real line), then well-behaved pseudo-asymptotic TOV systems are genuine TOV systems in that domain, ensuring the existence of new fourteen analytic solutions for extended TOV equations. The solutions are classified according to the nature of the matter (ordinary or exotic) and to the existence of cavities and singularities. It is shown that at least three of them are realistic, in the sense that they are formed only by ordinary matter and contain no cavities or singularities.
\end{abstract}

\section{Introduction}
Since the seminal work of Chandrasekhar, the axiomatization problem of astrophysics has been neglected. In \cite{TOV2} the authors reintroduced the problem, showing that arbitrary clusters of stellar systems experience fundamental constraints. In this paper we continue this work, focusing on a specific class of stellar systems: the TOV systems.

In order to state precisely the question being considered and our main results, we recall some definitions presented in  \cite{TOV2}. A \emph{stellar system} of degree $(k,l)$ is a pair $(p,\rho)$ of real functions, with $p$ piecewise $C^k$-differentiable and $\rho$ piecewise $C^l$-differentiable, both defined in some union of intervals $I\subset  \mathbb{R}$, possibly unbounded. We usually work with systems that are endowed with an additional piecewise $C^m$-differentiable function $M:I\rightarrow\mathbb{R}$, called \emph{mass function}. In these cases we say that we have a system of degree $(k,l,m)$. The space of all these systems is then
$$C_{pw}^{k}(I)\times C_{pw}^{l}(I)\times C_{pw}^{m}(I) \equiv C_{pw}^{k,l,m}(I), $$
where $C_{pw}^{r}(I)$ denotes the vector space of piecewise $C^r$-differentiable functions on $I$. A vector subspace $\mathrm{Stellar}^{k,l,m}_I$ of $C_{pw}^{k,l,m}(I)$ is called a \emph{cluster} of systems of degree $(k,l,m)$. Let $\mathrm{Stellar}^{l,m}_I$ denote the subspace obtained from the cluster by forgetting the variable $p$.

We say that a system with mass function $(p,\rho,M)$ is a \emph{TOV system} if the Tolman--Oppenheimer--Volkoff (TOV) equation holds
\begin{eqnarray}
\label{tov1}
p'(r) & =- &G \frac{\Bigl(\rho(r)+\frac{p(r)}{c}\Bigr)\Bigl(M(r)+4\pi r^{3}\frac{p(r)}{c^{2}}\Bigr)}{r^{2}\bigl(1-\frac{2M(r)G}{c^{2}r}\bigr)}
\end{eqnarray}
where $G$ and $c$ are respectively Newton's constant and the speed of light, which from now on we will normalize as $G=1=c$ (see \cite{weinberg}).

In a general cluster $\mathrm{Stellar}^{k,l,m}_I$ we define a \emph{continuity equation} as an initial value problem $M'=F(M,\rho,r)$ in $\mathrm{Stellar}^{l,m}_I\times I$ with initial condition $M(a)=M_a$. Since the vector space of piecewise differentiable functions is locally convex but generally neither Banach nor Fr\'{e}chet \cite{TOV2,generalized_norm}, the problem of general existence of solutions for continuity equations is much more delicate \cite{ODE_1,ODE_2}. Therefore one generally works with continuity equations which are integrable. The classical example is
\begin{equation}\label{continuity_equation}
M'=4\pi r^2\rho.
\end{equation}

Notice that if a system of degree $(k,l,m)$ is endowed with (\ref{continuity_equation}), then $m=l+1$. We define a \emph{classical TOV system} as a TOV system equipped with the classical continuity equation (for some initial condition). Let $\mathrm{TOV}^{k,l}_I$ be the cluster of stellar systems of degree $(k,l,l+1)$ generated by (i.e, the linear span of) the classical TOV systems. We can now state the main problem:
\begin{problem}[classification of TOV clusters]
Given $I$, $k$ and $l$, determine the structure of $\mathrm{TOV}^{k,l} _I$ as a subspace of $C_{pw} ^{k,l,l+1}(I)$ for some fixed locally convex topology.
\end{problem}

Since  $(0,0)\in \mathrm{TOV}^{k,l}_I$, this set is non-empty. Thus, the simplest thing one can ask about it is if it is nontrivial. This means asking if for every given $k,l$ there exists at least one nonzero pair $(p,\rho)$ for which equation (\ref{tov1}) is satisfied when assuming (\ref{continuity_equation}). Before answering this question, notice that by means of isolating $\rho$ in (\ref{continuity_equation}) and substituting the expression found in (\ref{tov1}) we see that the TOV equation is of Riccati type \cite{generating_TOV,generating_TOV_2}:
\begin{equation}
p'(r)=A(r)+B(r)p(r)+C(r)p(r)^{2}
\label{riccati}
\end{equation}
where
\begin{eqnarray}
\label{a1} A(r) & = & -\frac{M(r)M'(r)}{4\pi r^{4}\bigl(1-2\frac{M(r)}{r}\bigr)},\\
\label{b1} B(r) & = & -\frac{\bigl(\frac{M'(r)}{r}+\frac{M(r)}{r^{2}}\bigr)}{1-2\frac{M(r)}{r}},\\
\label{c1} C(r) & = & -\frac{4\pi r}{1-2\frac{M(r)}{r}}.
\end{eqnarray}
Therefore, it is a nonlinear and non autonomous equation which, added to the fact that $C_{pw} ^{k,l+1} (I)$ is only a locally convex space (so that the general existence theorems of ordinary differential equations do not apply), makes it hard to believe that $\mathrm{TOV}^{k,l} _I$ has a nontrivial element. Even so, if  $\rho$ is a constant function, then (\ref{riccati}) becomes integrable, showing that $\mathrm{TOV}^{k,l} _I$ is at least one-dimensional \cite{weinberg}. There are some results \cite{generating_TOV,generating_TOV_2} allowing that, under certain conditions, a solution of a TOV equation can be deformed into another solution, which may imply a higher dimension.

\begin{remark}
\emph{We point out that the generating theorem $(P1)$ proved in \cite{generating_TOV,generating_TOV_2} implies that the TOV cluster is invariant under} specific perturbations\emph{. So, a natural question is if there exist $I,k,l$ such that $\mathrm{TOV}^{k,l} _I$ is invariant under} arbitrary \emph{perturbations. This is clearly false, because this would imply that any stellar system is TOV. Instead, we can ask about invariance by arbitrary} small \emph{perturbations. Again, we assert that there are good reasons to believe that the answer is negative. Indeed, when we say that a subset of a topological space is ``invariant by small perturbations'' we are saying that it is actually an open subset. Thus, assuming invariance, we are saying that $\mathrm{TOV}^{k,l} _I$  is an open subset of  $C_{pw} ^{k,l+1} (I)$ in our previously fixed locally convex topology. Now assume that the locally convex topology is Hausdorff and that $\mathrm{TOV}^{k,l} _I$ is finite-dimensional. Then the TOV cluster is also a closed subset \cite{top_vec_space_1,top_vec_space_2}. Since topological vector spaces are contractible \cite{top_vec_space_1,top_vec_space_2}, they are connected and therefore a subset which is both open and closed must be empty or the whole space. But we know that $\mathrm{TOV}^{k,l} _I$ is at least one dimensional and does not coincide with $C_{pw} ^{k,l+1} (I)$. In sum, if there exists some locally convex topology in which the TOV cluster is invariant by small perturbations, then the following things cannot hold simultaneously:}
\begin{itemize}
    \item \emph{$\mathrm{TOV}^{k,l} _I$ is finite-dimensional;}
    \item \emph{the locally convex topology is Hausdorff.}
\end{itemize}
\emph{However, both conditions are largely expected to hold simultaneously, leading us to doubt the existence of a topology making $\mathrm{TOV}^{k,l} _I$ an open set}\footnote{Actually, this is not a special property of the TOV equation, but a general behavior of the solution space of elliptic differential equations \cite{elliptic_1}. The fact that TOV systems are modeled by an elliptic equation will be explored in a work in preparation.}.
\end{remark}

Another way to get information about the TOV cluster is not to look at the TOV cluster directly, but to analyze its behavior in some regime. For instance, if in (\ref{tov1}) we expand $1/c$ in a power series and discard all terms of higher order in $c$ (which formally corresponds to taking the limit $c\rightarrow \infty$) we get a new equation approximately describing the TOV equation. Therefore, studying the cluster $\mathrm{Newt}^{k,l}_I$ of stellar systems satisfying this new equation (called \emph{Newtonian systems}) we are getting some information about the original TOV cluster. In fact, if we consider the subspace $\mathrm{Poly}^{k,l}_I \subset \mathrm{Newt}^{k,l}_I$ of Newtonian systems satisfying an additional equation $p=\gamma^{q}\rho+a$, with $\gamma,c\in\mathbb{R}$ and $q\in\mathbb{Q}$, then the Newtonian equation becomes a Lane-Emden equation, which has at least three independent solutions (besides the constant ones), showing that $\mathrm{Newt}^{k,l}_I$ is at least four-dimensional \cite{weinberg,biblia_politropes}.

In this article we analyze the structure of $\mathrm{TOV}^{k,l}_I$ in a limit other than the Newtonian one: we work in the \emph{pseudo-asymptotic limit}. Before saying what this limit is, let us first say what it is not. We could think of defining a ``genuine'' asymptotic limit of TOV by taking the limit $r\rightarrow \infty$ in the TOV equation (\ref{tov1}) in a similar way as done for getting the Newtonian limit, trying to obtain a new equation. In doing this we run into two obstacles:

\begin{enumerate}
\item differently of  $c$ (which is a parameter), $r$ is a variable. Therefore, when taking the limit $r\rightarrow \infty$ we have to take into account the $r\rightarrow \infty$ behavior of all functions depending on $r$;
\item in order to get a new equation we have to fix boundary conditions for the functions, loosing part of the generality.
\end{enumerate}

Thus, for us, \emph{pseudo-asymptotic limit is not the same as a genuine asymptotic limit}. Another approach would be to work with an additional differential equation in $C^{l+1}_{pw}(I)\times I$, called a \emph{coupling equation} and given by

\begin{equation}{\label{coupling_equation}
\Lambda(M,M',M^{(2)},...,M^{(s)},r)=0, }
\end{equation}
where $M^{(i)}$ denotes the $i$-th derivative of $M$ and $1\leq s \leq l+1$ is the order of the equation. The function $\Lambda$ itself is called the \emph{coupling function} that generates equation (\ref{coupling_equation}). Let us consider the space $\mathrm{Ind}_{\Lambda}^{l+1}(I)$  of all   $C^{l+1}$-differentiable mass functions $M$ such that, if they are solutions of the coupling equation defined by $\Lambda$, then the corresponding TOV equation (\ref{riccati}) is integrable. So, for every coupling function $\Lambda$ we have $\mathrm{Ind}_{\Lambda}^{l+1}(I)\subset \mathrm{TOV}^{k,l} _I$ motivating us to define some kind of ``indirect'' asymptotic limit in TOV by taking the genuine asymptotic limit in (\ref{coupling_equation}). Again, we will have the two problems described above, but now the lack of generality is much less problematic, since we only need to consider boundary conditions for the single function $M$ and its derivatives. Even so, \emph{the pseudo-asymptotic limit is not the same as the indirect asymptotic limit}.

Let us now explain what we mean by pseudo-asymptotic limit. A \emph{decomposition} for the coupling equation (\ref{coupling_equation}) generated by $\Lambda$ is given by two other coupling functions $\Lambda_0$ and $\Lambda_1$  such that $\Lambda=\Lambda_0 + \Lambda_1$. We say that a decomposition is \emph{nontrivial} if $\Lambda_i\neq0$ for $i=0,1$. A \emph{split decomposition} of $\Lambda$ is a nontrivial decomposition such that $\Lambda_0$ generates a linear equation and $\Lambda_1$ generates a nonlinear equation. We say that a split decomposition is \emph{maximal} when both $\Lambda_0$ and $\Lambda_1$ do not admit split decompositions. Not all coupling functions admit a maximal split decomposition, e.g, when $\Lambda$ is linear. When $\Lambda$ admits such a decomposition we will say that it is \emph{maximally split}. So, let $\Lambda$ be a maximally split coupling function. The \emph{pseudo-asymptotic limit} of (\ref{riccati}) relative to $\Lambda$ is obtained by taking the genuine asymptotic limit in the nonlinear part $\Lambda_1$, added of the boundary condition $\lim _{r\rightarrow \infty}  \Lambda_1 = 0$, and maintaining unchanged the linear part $\Lambda_0$. In this case, the equation replacing (\ref{tov1}) is that generated by $\Lambda_0$. This new equation can be understood as a formal ``pseudo-limit'' of $\Lambda$, defined by $$\underset{r\rightarrow \infty}{\operatorname{psdlim}}\,\Lambda=\Lambda_0 + \lim_{r\rightarrow \infty}\Lambda_1.$$

For a given maximally split coupling function $\Lambda$, let $\mathrm{Psd}_{\Lambda}^{l+1}(I)$ denote the subspace of all mass-functions $M$ such that $$\lim_{r\rightarrow \infty} \Lambda_1 (M(r), M'(r),...,M^{(s)},r)=0,$$ i.e, which belong to the boundary conditions for $\Lambda$, and such that they are solutions of the pseudo-limit equation $\operatorname{psdlim}_{r \rightarrow \infty}\,\Lambda=0$. We can now state our main result.

\begin{theorem}\label{thm_introduction}
There exists $I$ such that for every $l$ there exists at least one maximally split coupling function such that $\mathrm{Psd}_{\Lambda}^{l+1}(I)$ is at least eleven-dimensional.
\end{theorem}

Generally, when we take a limit in a equation, the solutions of the newer equation are not solutions of the older one. This is why we cannot use the existence of Newtonian systems to directly infer the existence of new TOV systems. But the pseudo-asymptotic limit is different, precisely because it is not a formal limit. Indeed, suppose $M\in \mathrm{Psd}_{\Lambda}^{l+1}(I)$. Then by hypothesis it satisfies $\Lambda_0$. If in addition it satisfies $\Lambda_1$ (instead of obeying only the boundary condition), then it satisfies $\Lambda$ and, therefore, it belong to $\mathrm{TOV}^{k,l} _I$. We will show that by means of modifying $I$ in Theorem \ref{thm_introduction}, many of the mass functions will actually satisfy $\Lambda_1$, ensuring the existence of new integrable TOV systems. More precisely, we will show that if a pseudo-asymptotic system has a well-behaved extension to the extended real line $\overline{R}$, then it is actually a \emph{extended TOV system}, i.e, a TOV system which instead of being defined in a union $I$ of intervals of $\mathbb{R}$, its mass-function, density and pressure are all defined in $\overline{R}$. In sum, we have the following corollary:

\begin{corollary}\label{corollary_introduction}
The space $\mathrm{\overline{\overline{TOV}}}^{\infty,\infty}$ of piecewise $C^\infty$-differentiable extended TOV systems is at least eleven-dimensional.
\end{corollary}

The proof of Theorem \ref{thm_introduction} and its extrapolation to extended TOV systems will be done in Section \ref{proof} and in Section \ref{extending}, respectively, by making use of purely analytic arguments. In Section \ref{beyond} we argue that $\mathrm{Psd}_{\Lambda}^{\infty}(I)$ must have a higher dimension and we give two strategies that can be used to verify this. In Section \ref{sec_definitions} we present a physical classification for the densities associated with mass functions in Theorem \ref{thm_introduction} and, therefore, for the new extended TOV systems. We classify them according to if they possess or not cavities/singularities and to if they are composed of ordinary or exotic matter. In the process of classifying them we prove that they generally admit a topology change phenomenon (similar to a phase transition), allowing us to improve Theorem \ref{thm_introduction} and Corollary \ref{corollary_introduction}, showing that $\mathrm{\overline{\overline{TOV}}}^{\infty,\infty}$ is at least fifteen-dimensional.  We finish the paper in Section \ref{concluding_remarks} with a summary of the results.

\section{Proof of Theorem \ref{thm_introduction}} \label{proof_theorem}

Before giving the proof of Theorem \ref{thm_introduction}, let us emphasize that finding integrable TOV systems is a nontrivial problem. Notice that the simplest way to find explicit solutions of the Riccati equation is by choosing coefficients satisfying one of the following conditions:

\begin{enumerate}
\item $C$ identically zero. In this case the equation becomes a linear homogeneous first order ODE, which is separable.

\item $A$ identically zero. Then the equation is of Bernoulli type and, therefore, integrable by quadrature.

\item There are constants $a, b, c \in \mathbb{R}$ such that $A(r) = a$, $B(r) = b$ and $C(r) = c$ simultaneously. In this case, the equation is separable.
\end{enumerate}

\noindent Thus, recalling that TOV systems are described by a Riccati equation,  we could think of applying some of these conditions in (\ref{riccati}). It happens that the coefficients of this Riccati equation are not independent, but rather satisfy the conditions

\begin{equation}\label{coefrel}
\frac{4\pi r^4}{M(r)M'(r)} A(r) = \frac{r^2}{rM'(r) + M(r)} B(r) = \frac{1}{4\pi r} C(r).
\end{equation}

\noindent Therefore, the first two possibilities are ruled out for making  (\ref{riccati}) trivial. Also, the third condition along with (\ref{c1}) implies that

\begin{equation}
M(r) = r/2 - 2\pi r^2/c
\end{equation}
\noindent and $M'(r)$ is the corresponding linear polynomial. However, in  (\ref{coefrel}) the above $M(r)$ makes the $A$ term a quadratic polynomial, and the $B$ term a rational function of degree 1, so the equality can never hold.

\subsection{The proof}\label{proof}
Now we will prove Theorem \ref{thm_introduction}.

\begin{proof}
We start by noticing that in \cite{Tiberiu} it was shown that if the coefficients of any Riccati equation satisfy additional differential or integral conditions, then the nonlinearity of the starting equation can be eliminated, making it fully integrable. Each class of conditions is parametrized by functions $f:I \rightarrow \mathbb{R}$ and real constants. Because the authors of \cite{Tiberiu} worked only with smooth Riccati equations, they assumed $f:I \rightarrow \mathbb{R}$ smooth. However, it should be noticed that in the general situation we may have $f\in C^m_{pw}(I;\mathbb{R})$ or $f\in C^{m+1}_{pw}(I;\mathbb{R})$, where $m$ is the least order of differentiability of the coefficients of the Riccati equation. Furthermore, from (\ref{a1}), (\ref{b1}) and (\ref{c1}) we see that $m=l$, where $l$ is the class of $M$.

In the following, we will use these additional equations in order to build maximally split coupling functions.  Precisely, one of the integral conditions presented in \cite{Tiberiu} is

\begin{equation}
A(r) = \frac{f(r) - \left[ B(r) + C(r) \left[ \int \frac{f(s)-B^{2}(s)}{2C(s)}\, ds - c_1 \right] \right]^{2}}{4C(r)},
\label{case1}
\end{equation}
under which a explicit solution for Riccati equation (\ref{riccati}) is given by

\begin{equation}
p(r) = \frac{e^{\int^r (B(s)+C(s)h(s)) \, ds}}{c_0-\int^r C(s)e^{\int^s (B(\phi)+C(\phi)h(\phi)) \, d\phi} \, ds} + \frac{h(r)}{2}
\label{0P}
\end{equation}

\noindent where $c_0$ is a constant of integration and
\begin{equation}
 h(r) = \int^{r} \frac{f(s)-B^{2}(s)}{2C(s)}\, ds - c_1.
\end{equation}
Let us fix $f(r) = B^{2}(r)+2g(r)C(r)$, where $g:I \rightarrow \mathbb{R}$ is some integrable function. We will also take $c_1 = 0$.  Let $h(r) = \int g(s)\, ds$. Using the coefficients of TOV equation (\ref{riccati}), (\ref{case1}) becomes
\begin{equation}
M'(r) + \frac{M(r)M'(r)}{2\pi r^{3}h(r)} - \left( \frac{2h'(r)}{h(r)} - \frac{1}{r} \right) M(r) + 2\pi r^{2}h(r) + \frac{rh'(r)}{h(r)} = 0,
\label{nolinear}
\end{equation}
which is a coupling equation of order 1. Notice that it is maximally split, with maximal splitting decomposition given by
\begin{align}
 \Lambda _0 (M,M',r)&= M'-\left( \frac{2h'(r)}{h(r)} - \frac{1}{r} \right) M + 2\pi r^{2}h(r) + \frac{rh'(r)}{h(r)} \label{decomposition_0}\\
 \Lambda _1 (M,M',r)&=  \frac{MM'}{2\pi r^{3}h(r)} \label{decomposition_1}.
\end{align}
Therefore, giving $M$ and $h$ such that  $\displaystyle\lim_{r \to \infty} \Lambda_1 (M,M',r) = 0 $, then $M$ will automatically belong to $\mathrm{Psd}_{\Lambda}^{k}(I)$. We found ten such pairs. They are obtained by regarding $h$ as solutions of the differential equation
\begin{equation}
F(r) = 2\pi r^3h^2(r) + r^2h'(r),
\label{int}
\end{equation}
for different $F:I\rightarrow \mathbb{R}$, as organized in Table \ref{table_appendix} of  Appendix \ref{appendix}. Only to illustrate the method, we will show how the first row is obtained. The other rows are direct analogues, only involving more calculations.

Recall that we are trying to find pairs $(M,h)$ such that $\operatorname{psdlimit}_{r\rightarrow \infty}\, \Lambda =0$, i.e, such that $\lim _{r\rightarrow \infty} \Lambda _1 =0$ and such that the coupling equation induced by $\Lambda _0$ is satisfied. Suppose such a pair was found. Then, by the linearity of $\Lambda _0$, they become related (up to addition of a constant) by
\begin{equation}
M(r) = \frac{h^2(r)}{r} \left(\int_1^r -\frac{s^2 \left(h'(s)+2 \pi  s h^2(s)\right)}{h^3(s)} \, ds+c_2\right)
\label{solutionlinear}
\end{equation}
where $c_2$ is a integration constant. Taking $F=0$ in  (\ref{int}) we see that it is solved for
\begin{equation}
h(r) = \frac{1}{\pi (r^2 + c_1 / \pi)},
\label{0h}
\end{equation}
so that (\ref{solutionlinear}) becomes
\begin{equation}\label{0M}
M(r) = \frac{c_2h^2(r)}{r},
\end{equation}
whose derivative is
\[
M'(r) = \frac{c_2\left(2rh'(r)h(r)-h^2(r)\right)}{r^2}.
\]
Thus $\Lambda _1$ is given by the following expression, which clearly goes to zero as $r\to\infty$.
\begin{equation*}\label{lambda_1_0}
	\Lambda_1(M,M',r) = \Lambda _1 (r)= \frac{c_2^2h^2(r)\left(2rh'(r)-h(r)\right)}{2\pi r^6} = -\frac{c_2^2 \left(5 \pi  r^2-c_1\right)}{2 \pi  r^6 \left(c_1-\pi  r^2\right){}^4}.
\end{equation*}
Defining $(M,h)$ by (\ref{0M}) and (\ref{0h}), they will clearly satisfy the desired conditions, showing that $M\in \mathrm{Psd}_{\Lambda}^{l}(I)$.

Now, notice that all involved functions are actually piecewise $C^{\infty}$, so that we can take $l$ arbitrary. On the other hand, the domains of the functions in Table \ref{table_appendix} are different, but we can restrict them to the intersection of the domains and then (since we are working with piecewise differentiable functions) extend all of them trivially to the starting $I$. This finishes the proof of Theorem \ref{thm_introduction}, except by the fact that Table \ref{table_appendix} contains \emph{ten} linearly independent mass functions instead of the \emph{eleven} ones stated in Theorem \ref{thm_introduction}. The one missing is just the well known constant density solution. For completeness, let us show that it can also be directly obtained from our method. Indeed, by taking  $f(r) = B^2(r)$ in (\ref{case1}),  writing $c=  8\pi ^2 c_1 /3 $, and using the coefficients of the TOV system, the coupling equation (\ref{nolinear}) becomes

\begin{equation*}
M'(r) + \frac{3M(r)M'(r)}{4\pi cr^{3}} + \frac{M(r)}{r} +\frac{4 \pi cr^{2}}{3} = 0,
\end{equation*}
which has solution $M(r) = 4\pi cr^3 /3 $, whose associated density is $\rho(r) = c$.
 \end{proof}

 \begin{remark}
 \emph{In the construction, the motivation of the definition of $h(r)$ in (\ref{0P}) is the control of the integral term. By control, we mean that the $h$ function is a integral of a integrable function $g$, freely chosen.}
 \end{remark}
 \begin{remark}
 A note on the differentiability of the pressure. \emph{From (\ref{0P}) we see that $p\in C^{k'}_{pw}(I)$, where $k'$ is the minimum between the differentiability class of $h$ and the class of $M$. The latter is $\infty$, as obtained above, so that $k'$ is the class of $h$. But looking at Table \ref{table_appendix} we see that the class of $h$ is $\infty$ in each of the cases. Thus, $p\in C^{\infty}_{pw}(I)$.}
 \end{remark}

\subsection{Extending}\label{extending}

In the last section we proved that there exists a subset $I\subset \mathbb{R}$, which can be regarded as a disjoint union of intervals, such that we have a maximally split coupling function $\Lambda$ of order 1, whose corresponding space of pseudo-asymptotic TOV systems $\mathrm{Psd}_{\Lambda}^{\infty}(I)$ is at least eleven-dimensional. As discussed in the Introduction, a pseudo-asymptotic TOV system does not need to be a TOV system. Here we will show that by means of modifying $I$ properly, i.e, by working on the extended real line, we can assume that some of the pseudo-asymptotic systems that we have obtained really define TOV systems.

Recall that if there exists a coupling function $\Lambda$ such that a pseudo-asymptotic mass function $M$ satisfies not only the linear part $\Lambda _0$ but also the nonlinear one $\Lambda _1$, then $M$ actually defines an integrable TOV system. So, our problem is to analyze when the pseudo-asymptotic $M$ obtained in the last section satisfies the differential equation $\Lambda_1 (M,M',r)=0$ for $\Lambda _1$ given by (\ref{decomposition_1}). We will give sufficient conditions on the general pseudo-asymptotic mass functions in order for this to happen. That these conditions are satisfied for our mass functions will be a consequence of their classification. The fundamental step is the following result from real analysis.
\begin{lemma}\label{lemma_extension}
Let $f: I \rightarrow \mathbb{R}$ be continuous at a point $a_0\in I$ and such that $f(x)\rightarrow 0$ when $x\rightarrow a_0$. Assume that there exists $\varepsilon>0$ such that one of the following conditions is satisfied:
\begin{enumerate}
    \item[\emph{c1)}] $f$ is non-negative and non-decreasing in $(a_0-\varepsilon,a_0]$;
    \item[\emph{c2)}] $f$ is non-positive and non-increasing in $[a_0,a_0+\varepsilon)$.
\end{enumerate}
Then there exists $0< \varepsilon' \leq \varepsilon$ such that $f$ is constant and equal to zero in $(a_0-\varepsilon',a_0]$ in the first case, and in $[a_0,a_0+\varepsilon')$ in the second case.
\end{lemma}
\begin{proof}
Because $f$ is continuous in $a_0$ and $f(x)\rightarrow 0$ when $x\rightarrow a_0$, we have $f(a_0)=0$. If $f$ satisfies the first condition, since $f$ is non-decreasing, it follows that if $x \leq y$ in $(a_0-\varepsilon,a_0]$, then $f(x)\leq f(y)$. This means that for every $x$ we have $f(x)\leq f(a_0)=0$. But $f$ is  non-negative in $(a_0-\varepsilon,a_0]$, hence we must have $f(x)=0$ in this interval. For the second condition, an analogous argument will give the result.
\end{proof}

Now, recall that we can extend the real line $\mathbb{R}$ in two different ways: by adding a point at infinity $\infty$ or by adding both $+\infty$ and $-\infty$. In the first case we have the \emph{projectively extended real line} $\hat{\mathbb{R}}$, while in the second one we have the \emph{extended real line} $\overline{\mathbb{R}}$. For the arithmetic construction of these objects, see \cite{extended_line_1}. Topologically, both spaces acquire natural Hausdorff compact topologies: $\hat{\mathbb{R}}$ is the one-point compactification of $\mathbb{R}$ and, therefore, is homeomorphic to the circle $\mathbb{S}^1$, while $\overline{\mathbb{R}}$ is the two-point compactification of $\mathbb{R}$ and has an order topology homeomorphic to $[-1,1]$ \cite{extended_line_2}. In this article we will use only $\overline{\mathbb{R}}$. One can think of the homeomorphism $\phi: \overline{\mathbb{R}}\rightarrow [-1,1]$ as a rescaling of $\overline{\mathbb{R}}$ that brings the infinities closer together. This is similar to the idea of compactification of spacetime used in the Penrose diagrams, with the difference that there the underlining topological space does not become actually compact.

Notice that any piecewise continuous function $f:I\rightarrow \mathbb{R}$ which is not oscillating in $\pm \infty$ admits an extension $\overline{f}$ to $\overline{\mathbb{R}}$, as follows. We first extend it to $\mathbb{R}$ by defining $\overline{f}(x)=0$ when $x\notin I$ and then take $\overline{f}(\pm \infty)=\lim_{x\rightarrow \pm \infty}f(x)$. By definition, this extension is continuous at $\pm \infty$ \footnote{This clearly does not mean that $\overline{f}$ is continuous in the whole domain $\overline{\mathbb{R}}$, only piecewise continuous.}. Now, recall that the domain of any function is in one-to-one correspondence with its graph. Furthermore, if the function is piecewise continuous, then this correspondence is actually a piecewise homeomorphism. Thus, if $f:I\rightarrow \mathbb{R}$ is any function as above, we have the commutative diagram below, where $\overline{\overline{f}}=\phi \circ \overline{f} \circ \phi ^{-1}$ is the rescaling of the extension $\overline{f}$.

\begin{equation}{\label{diagram} \xymatrix{\mathbb{R} \ \ar@{^(->}[r] & \overline{\mathbb{R}} \ar[r]_{\simeq}^{\phi} & [-1,1] \\
I \ \ar[u]^f \ar@{^(->}[r] & \overline{\mathbb{R}} \ar@{-->}[u]^{\overline{f}} \ar[r]_{\simeq}^{\phi} & [-1,1] \ar@{-->}[u]_{\overline{\overline{f}}} \\
\operatorname{graph}(f) \ar[u]^{\simeq} \ar@{-->}[r] & \operatorname{graph}(\overline{f}) \ar[u]^{\simeq} \ar@{..>}[r]_{\simeq}^{\varphi} & \operatorname{graph}(\overline{\overline{f}}) \ar[u]_{\simeq}}}
\end{equation}

From the above remarks and from Lemma \ref{lemma_extension} we get the following corollary:

\begin{corollary}\label{corollary_extending}
Let $\Lambda$ be a maximally split coupling function, $M$ a pseudo-asymptotic mass function for the coupling function $\Lambda$  and suppose that $\Lambda _1$ is not oscillating in $\pm \infty$. If there exists $0< \varepsilon \leq 2$ such that $\overline{\overline{\Lambda_1}}$ satisfies condition \emph{(c1)} (resp. \emph{(c2)}) of Lemma \ref{lemma_extension}, then there exists $\varepsilon '$ such that $\overline{\overline{\Lambda_1}}$ is zero in $(1-\varepsilon,1]$ (resp. in $[-1,-1+\varepsilon ))$.
\end{corollary}
\begin{proof}
By hypothesis $\Lambda _1$ is not oscillating in $\pm \infty$, so that the extension $\overline{\Lambda _1}$ exists and it is continuous at $\pm \infty$. The points $\pm \infty$ are mapped onto $\pm 1$ by $\phi$, so that $\overline{\overline{\Lambda_1}}$ is continuous at $\pm 1$. The result then follows from Lemma \ref{lemma_extension}.
\end{proof}

Notice that starting with the TOV equation we can extend it to $\overline{\mathbb{R}}$ and then rescale the infinities by working at $[-1,1]$ via $\phi$. All we have done in the previous section will work in the same way. In particular, when finding situations in which $\overline{\overline{\Lambda_1}}=0$ we are finding cases in which the pseudo-asymptotic solutions of the \textbf{extended} TOV equation is a genuine solution of that \textbf{extended} equation. Let $\mathrm{\overline{\overline{TOV}}}^{\infty,\infty}$ denote the space of such solutions which are piecewise smooth, i.e, the space of piecewise $C^\infty$-differentiable extended TOV systems.

\begin{proposition}\label{corollary_extending_2}
$\mathrm{\overline{\overline{TOV}}}^{\infty,\infty}$ is at least eleven-dimensional.
\end{proposition}
\begin{proof}
Let us consider the pseudo-asymptotic mass functions $M$ obtained in Section \ref{proof}, whose underlying density functions are in Table \ref{table_appendix}, so that $\Lambda_1$ is given by (\ref{decomposition_1}), which depends on a function $h$, also listed in Table \ref{table_appendix}. Writing (\ref{decomposition_1}) explicitly for each $(M,h)$, as presented in Appendix \ref{appendix_2},  we see that $\Lambda_1$ is not oscillating in $\pm \infty$, so that $\overline{\Lambda_1}$ and $\overline{\overline{\Lambda_1}}$ are well-defined. From the commutativity of diagram (\ref{diagram}) we can analyze the graph of  $\overline{\overline{\Lambda_1}}$ looking at the asymptotic behavior of the graph of $\Lambda_1$. As we see in Appendix \ref{appendix_2}, each $\Lambda_1$ either becomes non-negative and non-decreasing when $r\rightarrow +\infty$ or non-positive and non-increasing when $r\rightarrow -\infty$, which means that the corresponding $\overline{\overline{\Lambda_1}}$ have the same behavior in a neighborhood sufficiently small of $\pm 1$. The result follows from Corollary \ref{corollary_extending}.
\end{proof}

\begin{remark}[important remark]
\emph{Let  $\mathrm{TOV}^{\infty,\infty}_{I,0}\subset \mathrm{TOV}^{\infty,\infty}_I$ be the space of piecewise smooth TOV systems which are not oscillating in $\pm \infty$. We have an inclusion }

$$\overline{\overline{(\cdot)}}:\mathrm{TOV}^{\infty,\infty}_{I,0} \hookrightarrow \mathrm{\overline{\overline{TOV}}}^{\infty,\infty},\quad \mathrm{given \;by} \quad \overline{\overline{(p,M}})=(\overline{\overline{p}},\overline{\overline{M}}).$$
\emph{In Proposition \ref{corollary_extending_2} we got ten extended solutions. The remaining one is, again, the classic constant solution, now regarded as an extended solution via the inclusion above. The important fact to have in mind is that the reciprocal does not hold: an extended TOV solution when restricted to some interval $I$ of $\mathbb{R}$ is not necessarily a solution of the actual TOV equation in $I$. Indeed, if an extended TOV system $(p,M)$ depends explicitly on its behavior at $\pm$, then when restricting to $I$ the equation will not be preserved, so that $(p\vert _I,M\vert _I)$ will not belong to $\mathrm{TOV}^{\infty,\infty}_I$. Notice that this is exactly the situation of Corollary \ref{corollary_extending}, so that we cannot use Proposition \ref{corollary_extending_2} in order to get solutions of the actual TOV equation.}
\end{remark}

\subsection{Beyond}\label{beyond}
We proved Theorem \ref{thm_introduction} ensuring the existence of a coupling function $\Lambda$ whose space $\mathrm{Psd}_{\Lambda}^{\infty}(I)$ of pseudo-asymptotic solutions is \emph{at least} eleven-dimensional. In this section we show that it is at least fifteen-dimensional and discuss why it is natural to believe that it has an even higher dimension. The first assertion is due to the following reason:
\begin{enumerate}
   \item \emph{Existence of critical configurations exhibiting phase transitions}. Notice that the integrability conditions of \cite{Tiberiu}, such as (\ref{case1}), (\ref{case6}) and (\ref{case8}), depends on two constants $c_1$ and $c_2$. Consequently, the pseudo-asymptotic solutions of the corresponding coupling equations also depend on such constants. Let us write $M_{1,2}$ to emphasize that $M$ is a pseudo-asymptotic mass function depending on $c_1$ and $c_2$. Now, recall that two mass functions $M,N$ are linearly dependent in $\mathrm{Psd}_{\Lambda}^{l}(I)$ if there exists a real number $c$ such that $M(r)=cN(r)$ for every $r\in I$. This means that if the dependence of $M_{1,2}$ on $c_1$ (resp. $c_2$) is \textbf{not} in the form  $M_{1,2}=c_1 M_{2}$ (resp. $M_{1,2}=c_2 M_{1}$), then when varying $c_1$ (resp. $c_2$) we get at least two linearly independent pseudo-asymptotic mass functions. These linearly independent mass functions can be obtained by defining a new piecewise differentiable map $\mathcal{M}:I\times \mathbb{R}\times \mathbb{R} \rightarrow \mathbb{R}$, such that $\mathcal{M}(r,c_1,c_2)=M_{1,2}(r)$, and then by studying its critical points. In typical cases (as for those obtained in Section \ref{proof}, i.e, for those whose density function is in Table \ref{table_appendix}) the function $\mathcal{M}$ is piecewise a submersion. Therefore, the pre-images $\mathcal{M}^{-1}(c)$ are submanifolds $S_c$ of $\mathbb{R}^3$ and the linearly independent mass functions can be obtained by searching for topology changes in $S_c$ when $c$ vary (in a similar way we search for phase transitions in a statistical mechanics system). In the next section we will analyze the topology of the surface $S_c$ corresponding to rows 1, 2 and 7 of Table (\ref{table_appendix}), showing that they admit one, two and one topology changes, respectively. This means that Theorem \ref{thm_introduction} and Corollary \ref{corollary_extending_2} can be improved, giving the theorem below.

  \begin{theorem}
   There exists a maximally split coupling function $\Lambda$ such that $\mathrm{Psd}_{\Lambda}^{\infty}(I)$ is at least fifteen-dimensional. Furthermore,  $\mathrm{\overline{\overline{TOV}}}^{\infty,\infty}$ is at least fifteen-dimensional.
  \end{theorem}
  The second assertion is suggested by the following reason:
    \item \emph{Existence of other maximally split coupling functions}. Recall that our starting point to get $\Lambda$ in (\ref{nolinear}) was the integral equation (\ref{case1}) obtained in \cite{Tiberiu} which when satisfied induces a solution for the TOV equation. In \cite{Tiberiu}, besides (\ref{case1}), other nine integral/differential equations playing the same role are presented. Applying to these other nine equations a strategy analogous to that used in (\ref{case1}) to obtain (\ref{nolinear}), allows us to obtain new coupling equations. We recall that each integrability equation in \cite{Tiberiu} becomes parametrized by certain constants and by an arbitrary function $f$. By making a suitable choice of $f$ in the sixth and eighth cases of \cite{Tiberiu}, we see that the induced coupling equations coincide with (\ref{nolinear}). Explicitly, the sixth case is given by
    \begin{equation}\label{case6}
B(r) = \frac{f_3(r) -A(r)C(r) - C^2(r)\left[\int \frac{f_3(r)}{2C(r)}\, dr - c_7\right]^{2}}{2C(r)\left[\int \frac{f_3(r)}{2C(r)}\, dr - c_7\right]}
\end{equation}
where $f_3\in C^{k+1}_{pw}(I)$ and $c_7$ is a constant of integration. For the choice $f_3(s) = 2C(s)g(s)$, where $g \in C^{k+1}_{pw}(I)$, the induced coupling equation reduces to (\ref{nolinear}). Furthermore, the eighth case is
\begin{equation}\label{case8}
    B(r) = \frac{1}{f_4(r)}\left[C(r)\frac{d}{dr}\left(\frac{f_4(r)}{C(r)}\right)-\frac{f_{4}^{2}(r)}{2}-2A(r)C(r)\right]
\end{equation}
where $f_4\in C^{k+1}_{pw}(I)$. If we choice $f_4(r) = h(r)C(r)$ again we get (\ref{nolinear}). Therefore, in view of the methods developed in Section (\ref{proof}), equations (\ref{case6}) and (\ref{case8}) do not differ from (\ref{case1}). On the other hand, we could not find $f$ which makes the coupling equation induced by each of the other seven cases equals to (\ref{nolinear}). This does not means that they will produce new pseudo-asymptotic mass functions which will eventually define (via Corollary \ref{corollary_extending}) new extended TOV systems\footnote{Indeed, many of the induced maximally split coupling equation have a nonlinear part which is really high nonlinear.}. Even so, it suggests the possibility.
\end{enumerate}

\section{Classification}\label{classification} \label{sec_definitions}
So far we have focused on getting new integrable extended TOV systems. In this section we will work to give physical meaning to discovered systems. In order to do this, we propose a simple classification of general stellar systems in which we will consider the new extended TOV systems. Indeed, let $(p,\rho)\in C^k_{pw}(I)\times C^l_{pw}(I)$ be a stellar system of degree $(k,l)$. We say that it is
\begin{itemize}
\item \emph{ordinary (resp. exotic)} in an open interval $J \subset I$ if the density $\rho$ is positive (resp. negative)
in each point of $J$. That is, $\rho(r)>0$ (resp. $\rho(r)<0$) for $r$ in $J$;
\item \emph{without cavities} if $\rho$ has no zeros, i.e, $\rho(r)\neq 0$ for every $r$ in $I$;
\item \emph{without singularities} if its domain is an open interval $(0,R)$, where $R$ can be $\infty$;
\item \emph{smooth} if it is without singularities and $\rho \in C^{\infty}(I)$;
\item \emph{realistic} if it is smooth and ordinary in $I$.
\end{itemize}
A \emph{cavity radius} of $(p,\rho)$ is a zero of $\rho$. Similarly, a \emph{singularity} of $(p,\rho)$ is a discontinuity point of $\rho$. So, $(p,\rho)$ is without cavities (resp. without singularities) iff it has no cavity radius (resp. singularity).

When looking at Table \ref{table_appendix} it is difficult to believe that some of the stellar systems there described are realistic. In fact, as can be rapidly checked, when defined in their maximal domain, these systems are in fact unrealistic. But as we will see, when restricted to a small region, many of them becomes realistic. This becomes more clear looking at Figure \ref{drawfigures} below, which describe the classification of certain rows of Table \ref{table_appendix}. In the schematic drawings, the filling by the grid is associated to exotic matter, whereas the filling by the hexagons to ordinary matter. The dashed circles represents a singularity radius and the dot circles a cavity radius.

\begin{figure}[h]	
	\centering
	\begin{subfigure}[h]{0.25\textwidth}
		\centering
		\includegraphics[scale=0.25]{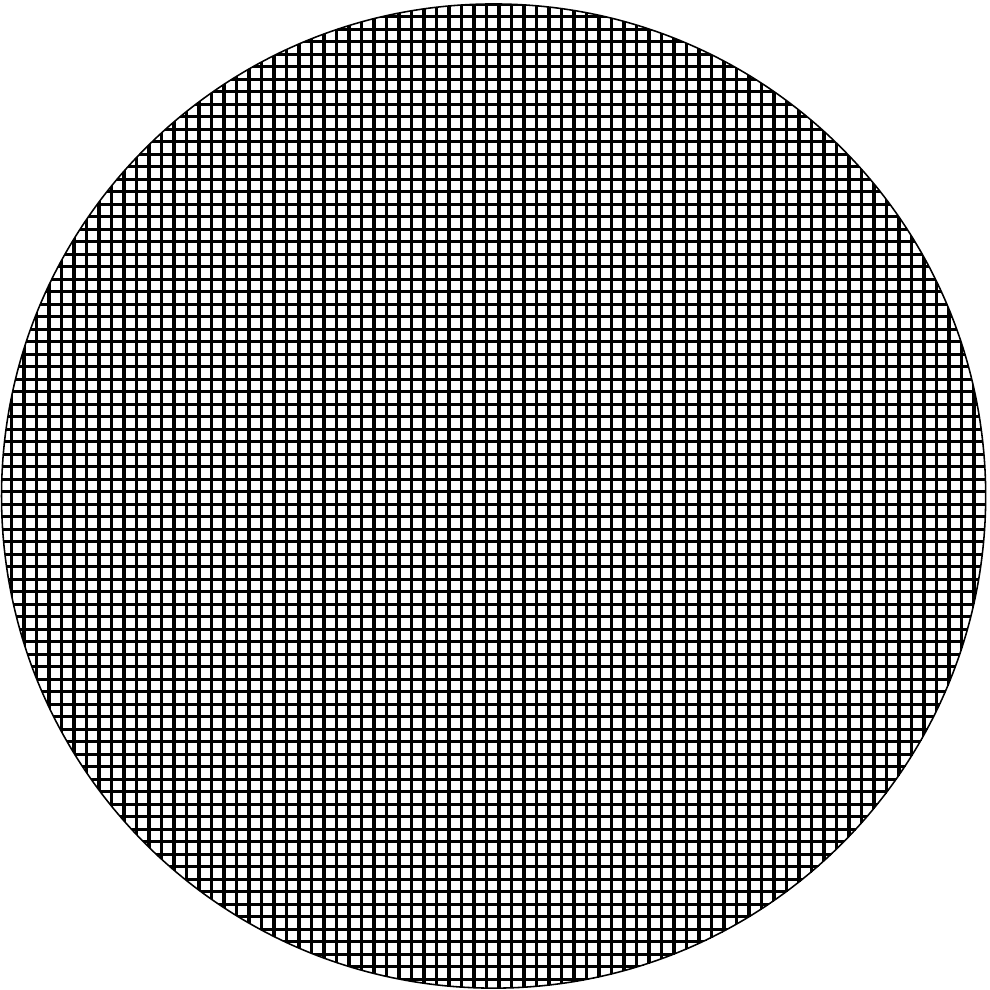}
		\caption{\centering Schematic Drawing of Figure \ref{graph1}}
        \label{draw1}		
	\end{subfigure}
	\quad\quad
	\begin{subfigure}[h]{0.25\textwidth}
	\centering
    \includegraphics[scale=0.25]{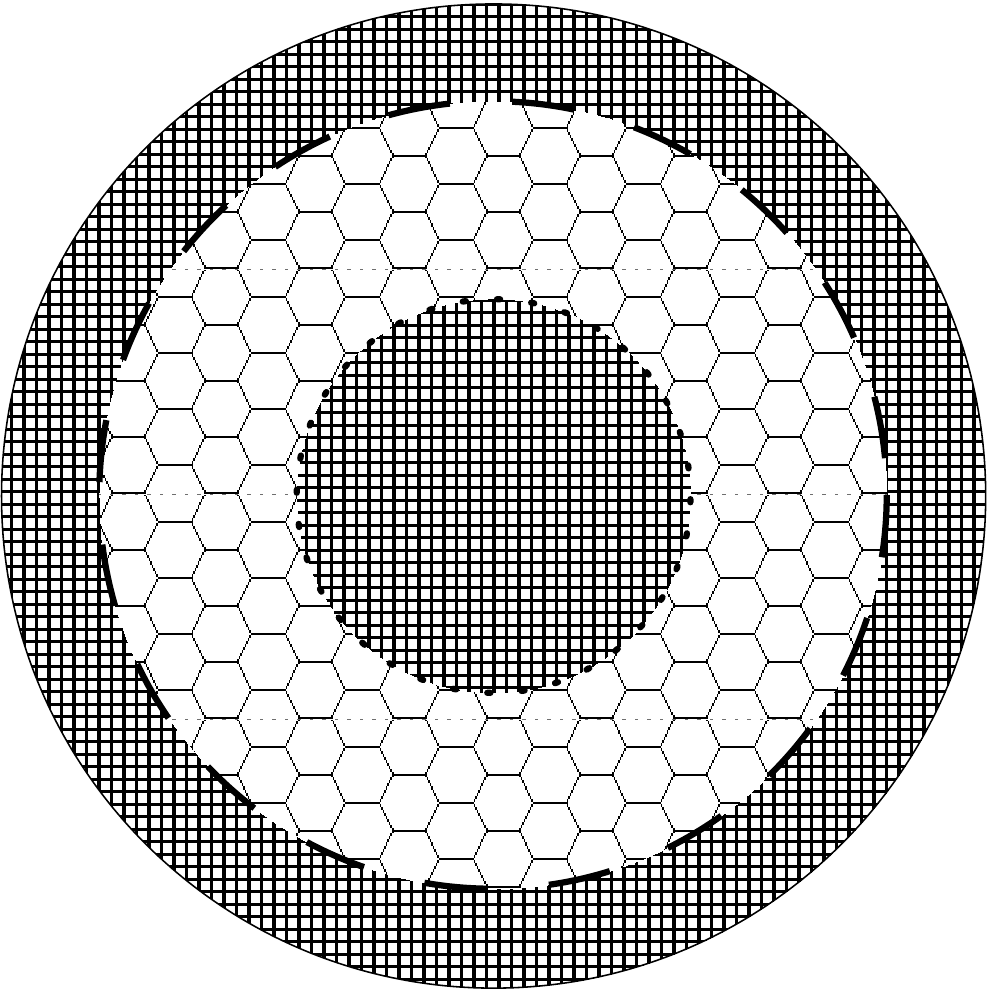}
    \caption{\centering Schematic Drawing of Figure \ref{graph2}}
    \label{draw2}
    \end{subfigure}
    \\
	\begin{subfigure}[h]{0.25\textwidth}
	\centering
	\includegraphics[scale=0.25]{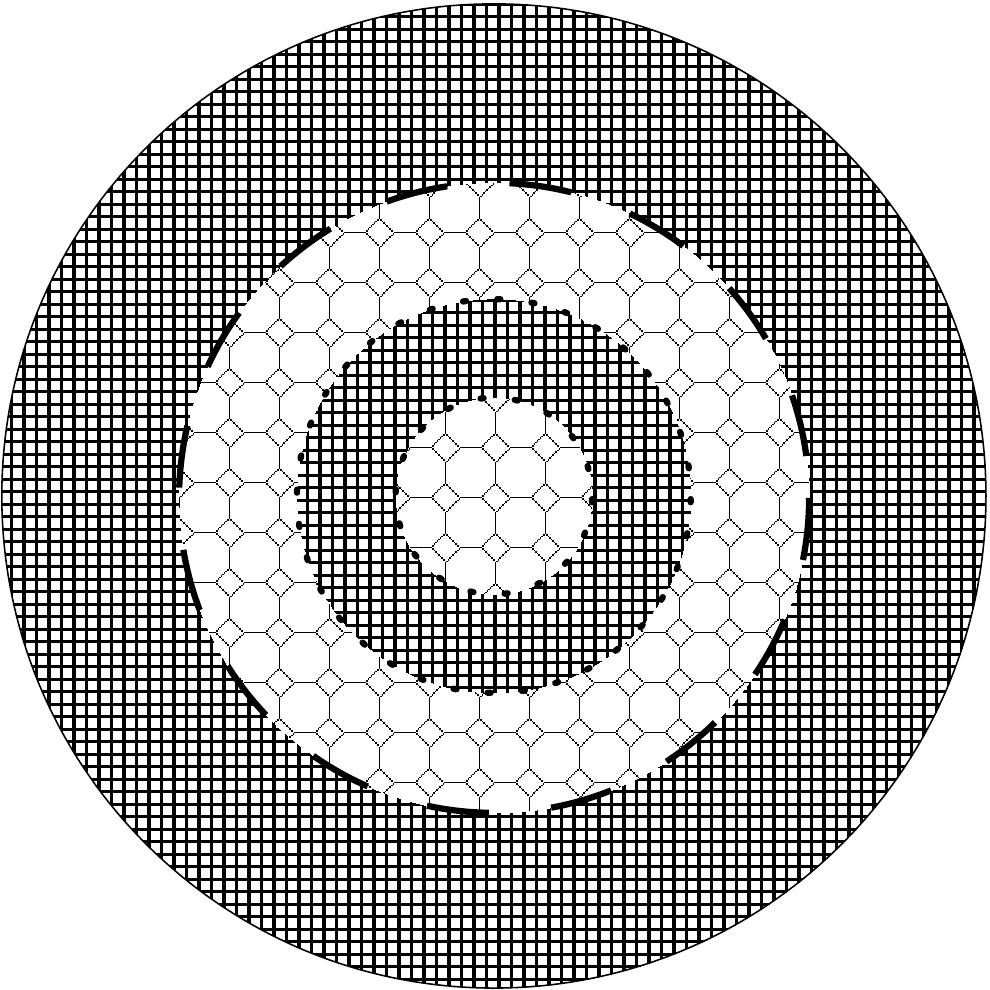}
	\caption{\centering Schematic Drawing of Figure \ref{graph6}}
    \label{draw3}	
	\end{subfigure}
	\quad\quad
    \begin{subfigure}[h]{0.25\textwidth}
	\centering
	\includegraphics[scale=0.25]{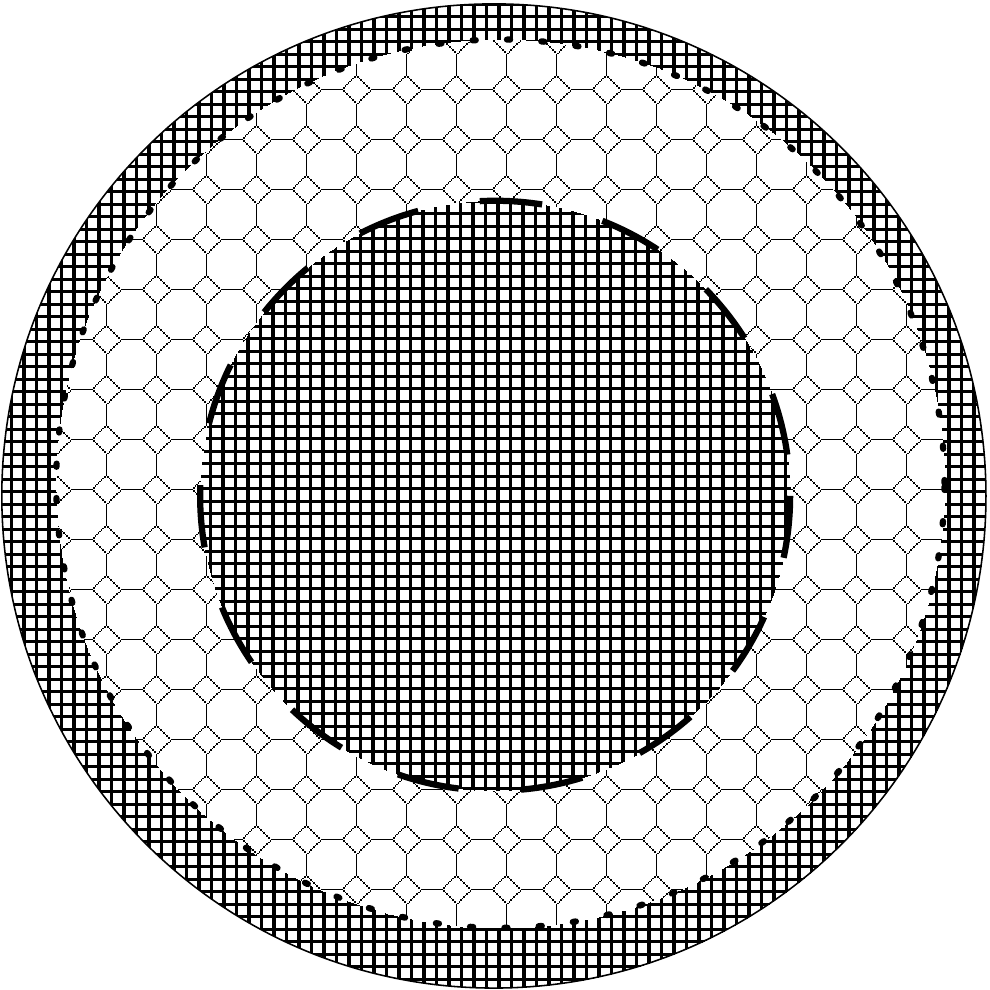}
	\caption{\centering Schematic Drawing of Figure \ref{graph7}}
    \label{draw4}	
	\end{subfigure}
	\\
	\begin{subfigure}[h]{0.25\textwidth}
	\centering
	\includegraphics[scale=0.25]{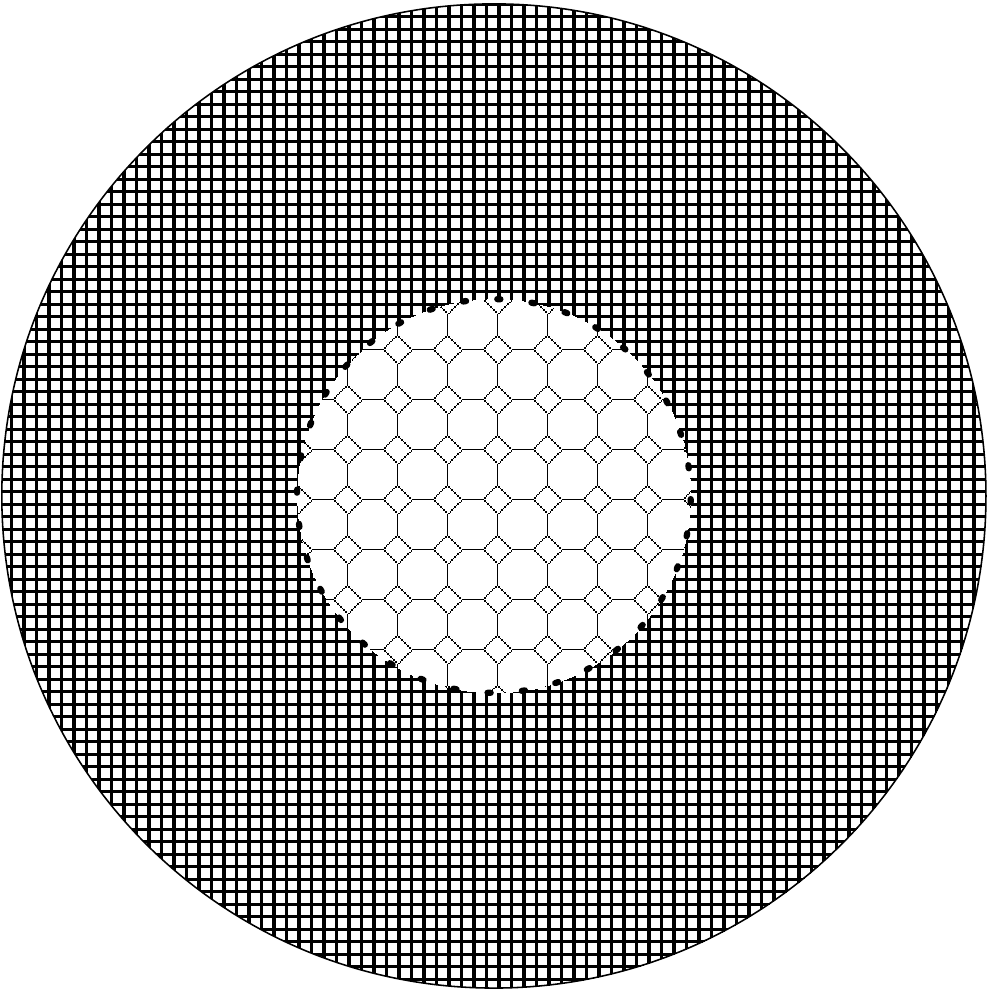}
	\caption{\centering Schematic Drawing of Figure \ref{graph8}}
    \label{draw5}		
	\end{subfigure}
	\quad\quad
	\begin{subfigure}[h]{0.25\textwidth}
	\centering
	\includegraphics[scale=0.25]{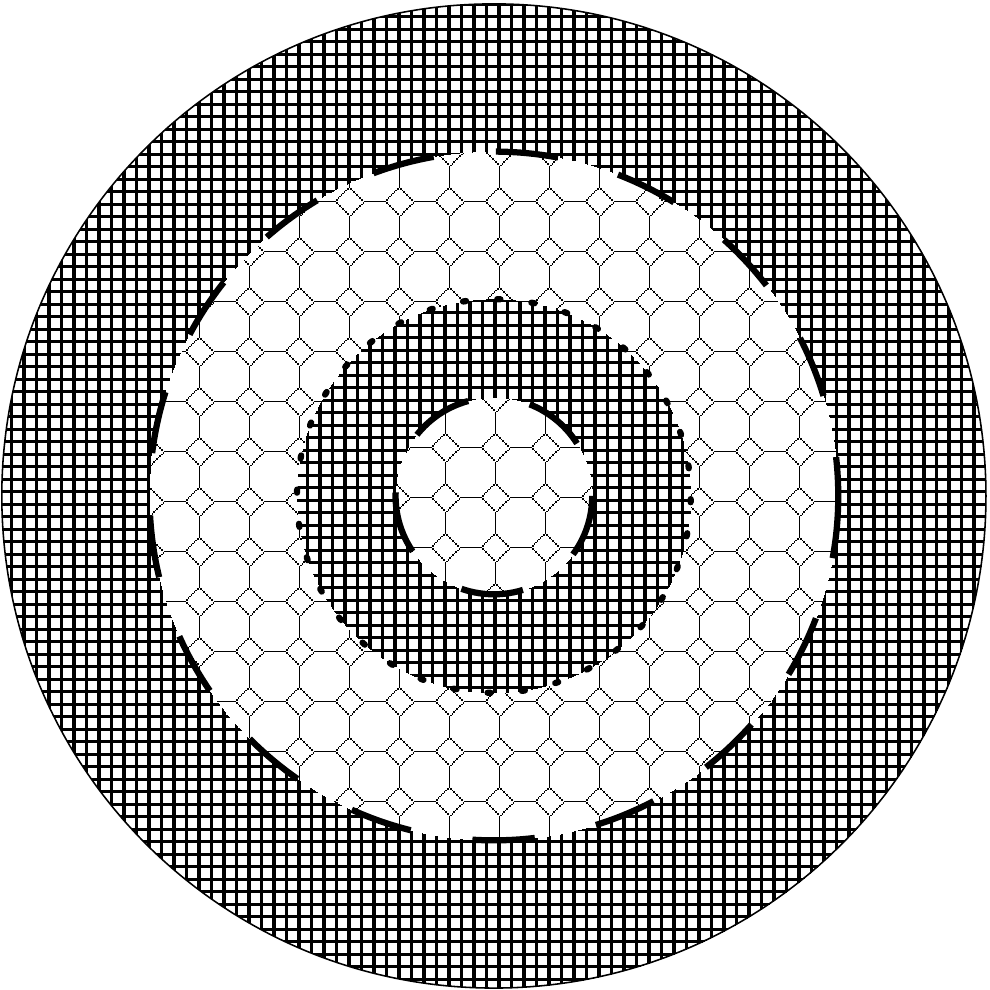}
	\caption{\centering Schematic Drawing of Figure \ref{graph9}}
    \label{draw6}	
	\end{subfigure}
    \caption{Classification of some new pseudo-asymptotic TOV systems.}
    \label{drawfigures}
\end{figure}

Notice that in order to do this classification we need to search for the zeros of the density function, which will give the radii in which there is no matter inside the star. Suppose that we found one of them, say $r_o$. If $\rho$ is continuous in that radius, then it is a cavity; otherwise, it is a singularity. The fundamental difference between them is that continuity implies that $\rho$ cannot change sign in neighborhoods of $r_o$. This means that a star containing only cavities is composed either of ordinary matter or of exotic matter. On the other hand, stellar systems with singularities may contain both ordinary and exotic matter.

The stellar systems considered in Table \ref{table_appendix} have densities of the form $\rho (r)=\frac{p(r)}{q(r)}\log(o(r))$, where $p$, $q$ and $o$ are polynomials with integer or fractional powers. Singularities are identified with zeros of $q$, while cavities are given by zeros of $p$ and $o-1$. The existence of real roots for a given polynomial is strongly determined by its coefficients and, in the present situation, the coefficients of $p$, $q$ and $o$ depend on two real parameters $c_1$ and $c_2$. We write $p_{12}$, $q_{12}$ and $o_{12}$ in order to emphasize this fact. We can then search for \emph{critical configurations}, in which a small change of $c_1$ and $c_2$ produces a complete modification of the system, as in a phase transition in statistical mechanics.
The critical configurations can be captured by defining new functions $$\mathcal{P},\mathcal{Q},\mathcal{O}:[0,\infty)\times \mathbb{R}\times \mathbb{R} \rightarrow \mathbb{R}$$ by $\mathcal{P}(r,c_1,c_2)=p_{12}(r)$, and so on, which are piecewise submersions. The solution of $\mathcal{Q}$ is then an algebraic submanifold $S_p \subset \mathbb{R}^3$, possibly with boundary, that completely determines the behavior of the singular set of $\rho _{12}$ when we vary $c_1$ and $c_2$. Similarly, the disjoint union of the solutions sets of $\mathcal{P}$ and $\mathcal{O}-1$ also defines an algebraic submanifold $S_{po}$ of $\mathbb{R}^3$ which determines the behavior of the cavities when we vary $c_1$ and $c_2$.

Notice that a point $(r,c_1,c_2)\in S_q$ is a critical singularity of $\rho_{12}$ iff it admits topologically distinct neighborhoods. Analogously, the critical cavities are given by points in $S_{po}$ with non-homeomorphic neighborhoods. Finally, the critical configurations of the stellar system with density $\rho_{12}$ are the points of $S_q \sqcup S_{po}$ which are critical singularities or critical cavities. A manifold defined by the inverse image of a function is locally homeomorphic to the graph of the defining function. This means that a neighborhood for $(r,c_1,c_2)$ is just a piece of the graph of $\mathcal{P}$, $\mathcal{Q}$ or $\mathcal{O}-1$. Since the topology of a graph changes only at an asymptote, zeros of $Q$ with fixed $(c_1,c_2)$ give the singularities, while zeros with fixed $(r,c_2)$ and $(r,c_1)$ will give the critical singularities, and similarly for critical cavities.

Having obtained singularities, cavities and critical configurations, the classification is completed by determining the kind of matter, which can be done from graphical analysis. In next sections we will apply this strategy to rows 1, 2 and 7 of Table \ref{table_appendix}. The row 1 will produce the schematic drawings (\ref{draw1}) and  (\ref{draw2}), while row 2 will produce (\ref{draw3}) and (\ref{draw4}), and row 7 will give (\ref{draw5}) and (\ref{draw6}).

\subsection{Row 1 of Table \ref{table_appendix}}

In this case the density function in given by
\begin{equation}\label{rho_1_2}
\rho(r,c_1,c_2)=\frac{c_2 \left(c_1-5 \pi  r^2\right)}{4 \pi  r^4 \left(\pi  r^2-c_1\right){}^3}.
\end{equation}
Notice that $c_2$ is a multiplicative constant, so that it will give linearly dependent solutions, leading us to fix $c_2=1$. This means that the singular set $S_q$ is a submanifold of $\mathbb{R}^2$, while the set of cavities $S_{po}$ is a submanifold of $\mathbb{R}^3$ of the form $S\times \mathbb{R}$, with  $S\subset \mathbb{R}^2$.
More explicitly, singularities are given by the solutions of the algebraic equation
\begin{equation}
4\pi r^4(\pi r^2 - c_1)^3 = 0,
\label{singularities_1}
\end{equation}
while cavities are determined by the solutions of
\begin{equation}
c_1 - 5\pi r^2 = 0
\label{cavities_1}
\end{equation}
Both equations depend explicitly of $c_1$. If $c_1 \leq 0$, the only singularity radius is the origin $r = 0$ and there are no cavities. An example of this behavior is given by Figure \ref{graph1}, for $c_1=0$:

\begin{figure}[h]
 \centering
 \includegraphics[scale=0.5]{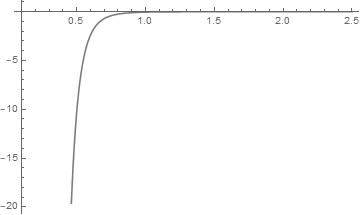}
    \caption{Plot of (\ref{rho_1_2}) for $c_1 = 0$ and $c_2 = 1$.}
    \label{graph1}
\end{figure}

It then follows that, for our choice $c_2=1$, the stellar system is composed only by exotic matter. On the other hand, if $c_1 > 0$, the singularities happen at $r = 0$ and at $r_0 =\sqrt{c_1/\pi}$, which is the non-negative solution of $\pi r^2 - c_1 = 0$.
The single cavity is given by the single positive root $r_1 = \sqrt{c_1/5\pi}$ of (\ref{cavities_1}), which lies in the interval $(0,r_0)$ bounded by the singularities. Thus, the type of matter inside $(0,r_0)$ may change, but it remains the same after crossing $r_0$. In order to capture this change of matter we analyze the sign of the derivative of $\rho _{12}$ at $r_1$. The derivative is
\begin{equation}
\rho _{12}'(r) =  \frac{c_2 \left(c_1^2-5 c_1\pi  r^2 + 10\pi^2 r^4\right)}{\pi  r^5 \left(\pi  r^2-c_1\right)^{4}}
\end{equation}
and we see that  $\rho _{12}'(r_1) > 0$. So, in that point, the star matter stops being exotic and becomes ordinary. Moreover, we have $\lim_{r\to r_{0+}} \rho _{12}(r) = +\infty$ and $\lim_{r\to r_{0-}} \rho _{12}(r) = -\infty$. It follows that in $(r_0,\infty)$ the star is composed of exotic matter. An example of this behavior is given by Figure \ref{graph2}, where $c_1=7$:

\begin{figure}[h]
    \centering
    \includegraphics[scale=0.5]{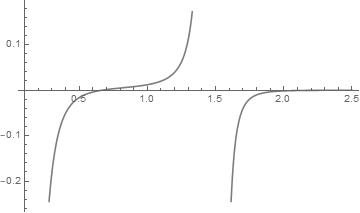}
    \caption{Plot of \ref{rho_1_2} for $c_1 = 7$ and $c_2 = 1$.}
    \label{graph2}
\end{figure}

Now, solving (\ref{cavities_1}) and (\ref{singularities_1}) we see that for each given $r_0$ there exists a single $c_1$ making $r_0$ a critical configuration.

\subsection{Row 2 of Table \ref{table_appendix}}

We start by noticing that in the present case the density function can be written as
\begin{equation}
\rho_{12}(r) = -\frac{\Gamma_{12}(r)}{8 \pi  r^4 \left(r^2-1\right) \left(-c_1+\pi  r^2+\pi  \log \left(r^2-1\right)\right){}^3}
\label{h'rho}
\end{equation}
where
\begin{align}
\Gamma_{12}(r) = &-2 c_1 c_2 \left(r^2-1\right)+3 \pi ^2 \left(r^2-1\right) \left(\pi  r^2-c_1\right) \notag \\
&\times \left[\log \left(r^2-1\right)\right]^2 - 2\pi [c_2 \left(1-5 r^2\right) r^2 \notag \\
&+ c_1^2 \left(r^4+2 r^2-1\right)] + \pi  \left(r^2-1\right) \notag \\
&\times \left[2 \pi  c_1 \left(1-3 r^2\right)+2 \left(c_1^2+c_2\right) +\pi ^2 \left(3 r^4-1\right)\right] \notag \\
&\times \log \left(r^2-1\right)-\pi ^2 c_1 \left(3 r^6-13 r^4+r^2+1\right) \notag \\
&+ \pi ^3 \left(r^2-1\right) \left[\log \left(r^2-1\right)\right]^3 \notag \\
&+\pi ^3 \left(r^8-r^6-5 r^4+r^2\right).
\label{h'gamma}
\end{align}
The function $\rho_{12}$ depends non-trivially on both variables $c_1$ and $c_2$ only in $\Gamma_{12}$, so that a priori its singular set and its set of zeros are arbitrary submanifolds of $\mathbb{R}^3$ and $\mathbb{R}^2$, respectively, and therefore we may expect many critical configurations.
The interval $[0,1]$ is clearly singular, while the singularities for $r>1$ are determined by solutions of the equation
\[
-c_1+\pi  r^2+\pi  \log \left(r^2-1\right) = 0.
\]
We assert that for each fixed $c_1$ there is precisely one solution. In other words, we assert that the singular set is diffeomorphic to $([0,1]\times \mathbb{R})\sqcup \mathbb{R}$.  Indeed, let $Y:[0,\infty) \times \mathbb{R} \rightarrow \mathbb{R}$ be the piecewise differentiable function
\begin{equation}
Y(r,c_1) = -c_1+\pi  r^2+\pi  \log \left(r^2-1\right).
\end{equation}
Clearly, there are $r_{0}$ and $r_{1}$ such that $Y(r_{0},c_1) < 0$ and $Y(r_{1},c_1) > 0$. But, the derivative of $Y(r,c_1)$ in the direction of $r$ is always positive for $r>1$. If fact, if $r>1$, then
\[
\frac{\partial Y}{\partial r}(r,c_1) = \frac{2r(r^2+\pi -1)}{r^2-1} > 0.
\]
Consequently, $Y$ has a single zero in $r>1$ corresponding to the unique singularity of $\rho _{12}$ in this interval. Moreover, for any $r_0 > 1$, we can set the constant $c_1$ appropriately so that $r_0$ corresponds the singularity. In fact, just set
\[
c_1 = \pi  r_0^2+\pi  \log \left(r_0^2-1\right).
\]

The cavities of $\rho_{12}$ in (\ref{h'rho}) are given by zeros of $\Gamma_{12}$. Differently of the previous case, we cannot isolate any of the variables, so that $S_{po}\subset \mathbb{R}^3$ will have complicated topology. We assert that the manifold of critical configurations is $\mathbb{R}^2$. Indeed, if we fix $c_1$ and vary $c_2$ we do not find any change of topology, while if we fix $c_2$, say $c_2=1$, we find two critical points, approximately at $c_1 = 1,7$ and at $c_1=4,6$, as shown in Figure \ref{h'figures}.

\begin{figure}[h]	
	\centering
	\begin{subfigure}[h]{0.3\textwidth}
		\centering
		\includegraphics[width=1.3in]{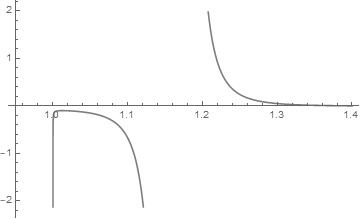}
		\caption{$c_1 = 1$}
        \label{graph3}		
	\end{subfigure}
	\quad
	\begin{subfigure}[h]{0.3\textwidth}
		\centering
		\includegraphics[width=1.3in]{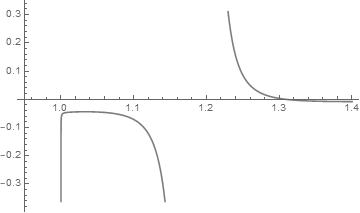}
		\caption{$c_1 = 1.6$}
        \label{graph4}	
	\end{subfigure} \\
    	\begin{subfigure}[h]{0.3\textwidth}
		\centering
		\includegraphics[width=1.3in]{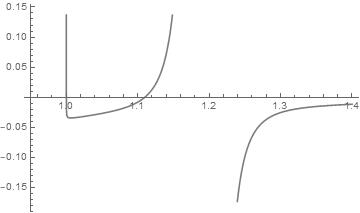}
		\caption{$c_1 = 1.8$}
        \label{graph5}			
	\end{subfigure}
	\quad
	\begin{subfigure}[h]{0.3\textwidth}
		\centering
		\includegraphics[width=1.3in]{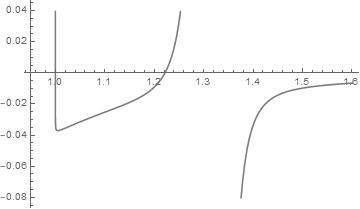}
		\caption{$c_1 = 4.5$}
        \label{graph6}	
	\end{subfigure}
	\quad
	\begin{subfigure}[h]{0.3\textwidth}
		\centering
		\includegraphics[width=1.3in]{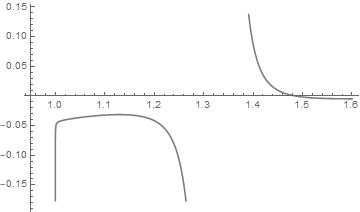}
		\caption{$c_1 = 4.7$}
        \label{graph7}	
	\end{subfigure}
    \caption{Plot of (\ref{h'rho}) for different values of $c_1$ and constant $c_2 = 1$.}
    \label{h'figures}
\end{figure}

\subsection{Row 7 of Table \ref{table_appendix}}

Previously we studied a stellar system with discrete critical configurations and another with a continuum of critical configurations. In both cases, graphical analysis were used to identify the critical configurations. Here we will discuss a third system, whose special feature is to show that even when we have an analytic expression for cavities and singularities, a graphical analysis is fundamental to complete the classification.

The density function is now of the form

\begin{equation}
\rho(r,c_1,c_2) = f(r,c_1,c_2)/q(r,c_1),
\label{h2rho2}
\end{equation}
where $f$ is a linear combination of polynomials and logarithmic functions, depending on both parameters $c_1$ and $c_2$. We will focus on singularities, so that only the expression of $q$ matters. It is the polynomial
\[
q(r,c_1) = 12\pi r^2(-c_1r + \pi r^3 +1)^3.
\]

Because $c_2$ does not affect $q$ we see that the singular set is a submanifold of $\mathbb{R}^2$. For each fixed $c_1$ the complex roots of $q$ can be written analytically with help of some software (we used Mathematica$\circledR$). They are $r_0=0$, with multiplicity two, and
\begin{equation}
r_1 = \frac{\sqrt[3]{2} \left(\sqrt{81 \pi -12 c_1^3}-9 \sqrt{\pi }\right){}^{2/3}+2 \sqrt[3]{3} c_1}{6^{2/3} \sqrt{\pi } \sqrt[3]{\sqrt{81 \pi -12 c_1^3}-9 \sqrt{\pi }}}
\label{root1}
\end{equation}
\begin{equation}
r_2 = \frac{\sqrt[3]{2} \sqrt[6]{3} \left(-1+i \sqrt{3}\right) \left(\sqrt{81 \pi -12 c_1^3}-9 \sqrt{\pi }\right){}^{2/3}-2 \left(\sqrt{3}+3 i\right) c_1}{2\ 2^{2/3} 3^{5/6} \sqrt{\pi } \sqrt[3]{\sqrt{81 \pi -12 c_1^3}-9 \sqrt{\pi }}}
\label{root2}
\end{equation}
\begin{equation}
r_3 = \frac{\sqrt[3]{2} \sqrt[6]{3} \left(-1-i \sqrt{3}\right) \left(\sqrt{81 \pi -12 c_1^3}-9 \sqrt{\pi }\right){}^{2/3}-2 \left(\sqrt{3}-3 i\right)c_1}{2\ 2^{2/3} 3^{5/6} \sqrt{\pi } \sqrt[3]{\sqrt{81 \pi -12 c_1^3}-9 \sqrt{\pi }}},
\label{root3}
\end{equation}
each of them with multiplicity three. Our stellar system then have singularities only for the values $c_1$ such that some of the radii above are real and non-negative. One can be misled to think that $r_1$ corresponds to a real zero iff $c_1 \leq \sqrt[3]{\frac{81\pi}{12}} \approx 2.77$. Although $c_1 \approx 2.77$ is clearly a critical configuration, a graphical and numerical analysis shows that $r_1$ is always real. In order to see this, we compare in Figure \ref{h2figures} the density function for $c_1 = 5 > 2.77$ and $c_1 = 1 < 2.77$.

\begin{figure}[h]	
	\centering
	\begin{subfigure}[h]{0.3\textwidth}
		\centering
		\includegraphics[width=1.5in]{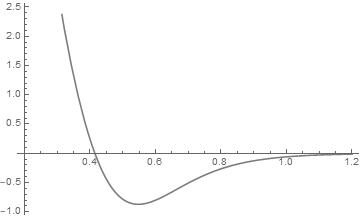}
		\caption{$c_1 = 1$}
        \label{graph8}		
	\end{subfigure}
	\quad
	\begin{subfigure}[h]{0.3\textwidth}
		\centering
		\includegraphics[width=1.5in]{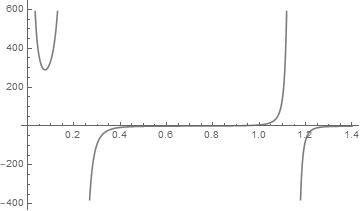}
		\caption{$c_1 = 5$}
        \label{graph9}	
	\end{subfigure}
    \caption{Plot of (\ref{h2rho2}) of $c_1$ and constant $c_2 = 1$.}
    \label{h2figures}
\end{figure}

Although Figure (\ref{graph8}) seems physically interesting, it is not:  a numerical evaluation shows that $r_1(c_1)  \approx -0.84$, so that despite being real, it is negative. In turn, Figure (\ref{graph9}) should appear strange: it contains two nonzero real roots while the analytic expressions above suggest that there is at most one non-null real root. Again, a numerical evaluation gives $r_1(c_1 = 5) \approx 1.15$, $r_2(c_1 = 5) \approx -1.35 + 1.69\times 10^{-21}i$ and $r_3(c_1 = 5) \approx 0.21 - 4.24\times 10^{-22}i$. So, we see that when $c_1$ grows, the imaginary parts of the roots $r_2$ and $r_3$ become increasingly smaller, allowing us to discard them. Thus, $c_1\approx 2.77$ determines a critical configuration such that for $c_1<2.77$ there are no singularities and for $c_1>2.77$ there are two of them.

\section{Summary} \label{concluding_remarks}
In this paper we considered the problem of classifying the stellar systems modeled by the piecewise differential TOV equation (\ref{tov1}). We began by introducing the problem in the general context presented in \cite{TOV2}, allowing us to formalize the problem as the determination of the structure of $\mathrm{TOV}^{k,l}_I$ as a subspace of $C^{k,l,l+1}_{pw}(I)$ relative to some locally convex topology. We showed that this subspace is generally not open if the topology is Hausdorff. We introduced another subspace $\mathrm{Psd}^{l+1}(I) \subset C^{k,l,l+1}_{pw}(I)$ of pseudo-asymptotic systems and we showed that for $l=\infty$ this space is at least fifteen-dimensional and that there are good reasons to believe that its dimension is larger. We then considered extended TOV systems, which are TOV systems defined on the extended real line $\mathbb{R}$, and we showed that if a pseudo-asymptotic systems has a nice behavior when extended to $\mathbb{R}$, then it is actually a extended TOV system, leading us to conclude that the space $\mathrm{\overline{\overline{TOV}}}^{\infty,\infty}$ of all of them is at least fifteen-dimensional. We presented a method of classification of these solutions, applying it to some of them, allowing us to show that there are new pseudo-asymptotic TOV systems  consisting only of ordinary matter and that contain no cavities or singularities.

\appendix
\section{Table of Solutions} \label{appendix}

\begin{table}[H]
\scalefont{0.6}

\begin{tabular}{c|c|l}

Function $F(r)$ (Eq. \ref{int}) & $h(r)$ & $\rho(r)$ \\ \hline

$0$ & $\frac{1}{\pi  r^2-c_1}$ & $\frac{c_2 \left(c_1-5 \pi  r^2\right)}{4 \pi  r^4 \left(\pi  r^2-c_1\right){}^3}$ \\ \hline

\multirow{6}{*}{$h'(r)$}  & \multirow{6}{*}{$\frac{1}{-c_1+\pi  r^2+\pi  \log \left(r^2-1\right)}$} & $[ -2 c_1 c_2 \left(r^2-1\right)+3 \pi ^2 \left(r^2-1\right) \left(\pi  r^2-c_1\right) \left[\log \left(r^2-1\right)\right]^2$ \\
& & $-2 \pi  \left(c_2 \left(1-5 r^2\right) r^2+c_1^2 \left(r^4+2 r^2-1\right)\right) + \pi  \left(r^2-1\right)$ \\
& & $\times \left(2 \pi  c_1 \left(1-3 r^2\right) + 2 \left(c_1^2+c_2\right)+\pi ^2 \left(3 r^4-1\right)\right) \log \left(r^2-1\right)$ \\
& & $-\pi ^2 c_1 \left(3 r^6-13 r^4+r^2+1\right)+\pi ^3 \left(r^2-1\right) \left[\log \left(r^2-1\right)\right]^3$ \\
& & $+\pi ^3 \left(r^8-r^6-5 r^4+r^2\right) ]:[ 8 \pi  r^4 \left(r^2-1\right)$ \\
& & $\times \left(-c_1+\pi  r^2+\pi  \log \left(r^2-1\right)\right){}^3 ]$ \\ \hline

\multirow{9}{*}{$rh'(r)$} & \multirow{9}{*}{$\frac{1}{-c_1+\pi  r^2+2 \pi  r+2 \pi  \log (r-1)-3 \pi }$} & $[ 60 \pi ^2 (r-1) \left(9 c_1+\pi  \left(8 r^3-9 r^2-18 r+31\right)\right) \log ^2(r-1)$ \\
& & $+30 \pi  \left(c_1^2 \left(4 r^4-r^3-3 r^2+17 r-11\right)-3 c_2 \left(5 r^3+r^2-5 r+3\right)\right)$ \\
& & $+2 \pi  (r-1) [-30 \pi  c_1 \left(8 r^3-9 r^2-18 r+31\right)-90 \left(c_1^2+c_2\right)$ \\
& & $+\pi ^2 \left(24 r^5+135 r^4-1460 r^3+390 r^2+1860 r-1669\right)]$ \\
& & $\times \log (r-1)-\pi ^2 c_1 [24 r^6+111 r^5-1595 r^4+1850 r^3+1470 r^2$ \\
& & $-3529 r+1669]+90 c_1 c_2 (r-1)+\pi ^3 [9 r^7-115 r^6-803 r^5$ \\
& & $+4465 r^4-3365 r^3-3529 r^2+5375 r-2037]$ \\
& & $-360 \pi ^3 (r-1) \log ^3(r-1)]:[360 \pi  (r-1) r^4 [-c_1+\pi  \left(r^2+2 r-3\right)$ \\
& & $+2 \pi  \log (r-1)]{}^3]$ \\ \hline

$h^2(r)$ & $\frac{r}{-c_1 r+\pi  r^3+1}$ & $\frac{-\pi  r^3 \left(10 c_1 r+15 c_2+7\right)+3 \left(-c_1 r+5 \pi  r^3-1\right) \log (r)+3 \left(c_1 \left(c_2+3\right) r+c_2-1\right)+2 \pi ^2 r^6}{12 \pi  r^2 \left(-c_1 r+\pi  r^3+1\right){}^3}$ \\ \hline

\multirow{3}{*}{$rh^2(r)$} & \multirow{3}{*}{$\frac{1}{-c_1+\pi  r^2-\log (r)}$} & $[-3 \pi  \left(c_1-1\right) r^4-\left(2 c_1^2-3 c_1+20 \pi  c_2+2\right) r^2$ \\
& & $+\left(\left(3-4 c_1\right) r^2+4 c_2-3 \pi  r^4\right) \log (r)+4 \left(c_1+2\right) c_2+\pi ^2 r^6$ \\
& & $-2 r^2 \log ^2(r)]:[-16 \pi  r^4 \left(c_1-\pi  r^2+\log (r)\right){}^3]$ \\ \hline

$r^2h^2(r)$ & $\frac{1}{-c_1+\pi  r^2-r}$ & $-\frac{\left(8 \pi  c_1-15\right) r^5-45 c_1 r^4-40 c_1^2 r^3+300 \pi  c_2 r^2-180 c_2 r-60 c_1 c_2+9 \pi  r^6}{240 \pi  r^4 \left(c_1-\pi  r^2+r\right){}^3}$ \\ \hline

$r^3h^2(r)$ & $\frac{2}{-2 c_1+2 \pi  r^2-r^2}$ & $\frac{7 (2 \pi -1) c_1 r^6-18 c_1^2 r^4+15 (1-2 \pi ) c_2 r^2+6 c_1 c_2-(1-2 \pi )^2 r^8}{12 \pi  r^4 \left((2 \pi -1) r^2-2 c_1\right){}^3}$ \\ \hline

\multirow{3}{*}{$\frac{h^2(r)}{r}$} & \multirow{3}{*}{$\frac{2 r^2}{-2 c_1 r^2+2 \pi  r^4+1}$} & $[\left(-2 c_1 r^2+2 \pi  r^4+1\right) \left(4 c_1 r^2+3 c_2 r^2+12 c_1 r^2 \log (r)-10 \pi  r^4+1\right)$ \\
& & $-2 \left(8 \pi  r^3-4 c_1 r\right) \left(c_2 r^3+4 c_1 r^3 \log (r)-2 \pi  r^5+r\right)]$ \\
& & $:[4 \pi  r^2 \left(-2 c_1 r^2+2 \pi  r^4+1\right){}^3]$ \\ \hline

$r^{1/2}h^2(r)$ & $\frac{\sqrt{r}}{-c_1 \sqrt{r}+\pi  r^{5/2}+2}$ & $\frac{-7 c_1^2 r^{3/2}-34 \pi  c_1 r^{7/2}-105 \pi  c_2 r^2+42 c_1 r+21 c_1 c_2+9 \pi ^2 r^{11/2}+126 \pi  r^3-84 \sqrt{r}}{84 \pi  r^{5/2} \left(-c_1 \sqrt{r}+\pi  r^{5/2}+2\right){}^3}$ \\ \hline

$r^{3/4}h^2(r)$ & $\frac{\sqrt[4]{r}}{-c_1 \sqrt[4]{r}+\pi  r^{9/4}+4}$ & $[440 c_1 r^{7/4}-525 \pi  c_2 r^{9/4}-118 \pi  c_1 r^4-45 c_1^2 r^2+105 c_1 c_2 \sqrt[4]{r}$ \\
& & $-210 c_2-1120 r^{3/2}+616 \pi  r^{15/4}+35 \pi ^2 r^6]$ \\
& & $:[420 \pi  r^{7/2} \left(-c_1 \sqrt[4]{r}+\pi  r^{9/4}+4\right){}^3]$ \\ \hline

\end{tabular}
\caption{ \label{table_appendix} Solutions of the coupling equation (\ref{nolinear}) in the pseudo-asymptotic limit.}
\label{solutions}
\end{table}

\section{Graphs of some of the $\Lambda_1$ associated to Table \ref{table_appendix}} \label{appendix_2}

\begin{figure}[H]	
	\flushleft
	\begin{subfigure}[h]{0.25\textwidth}
		\centering
		\includegraphics[scale=0.15]{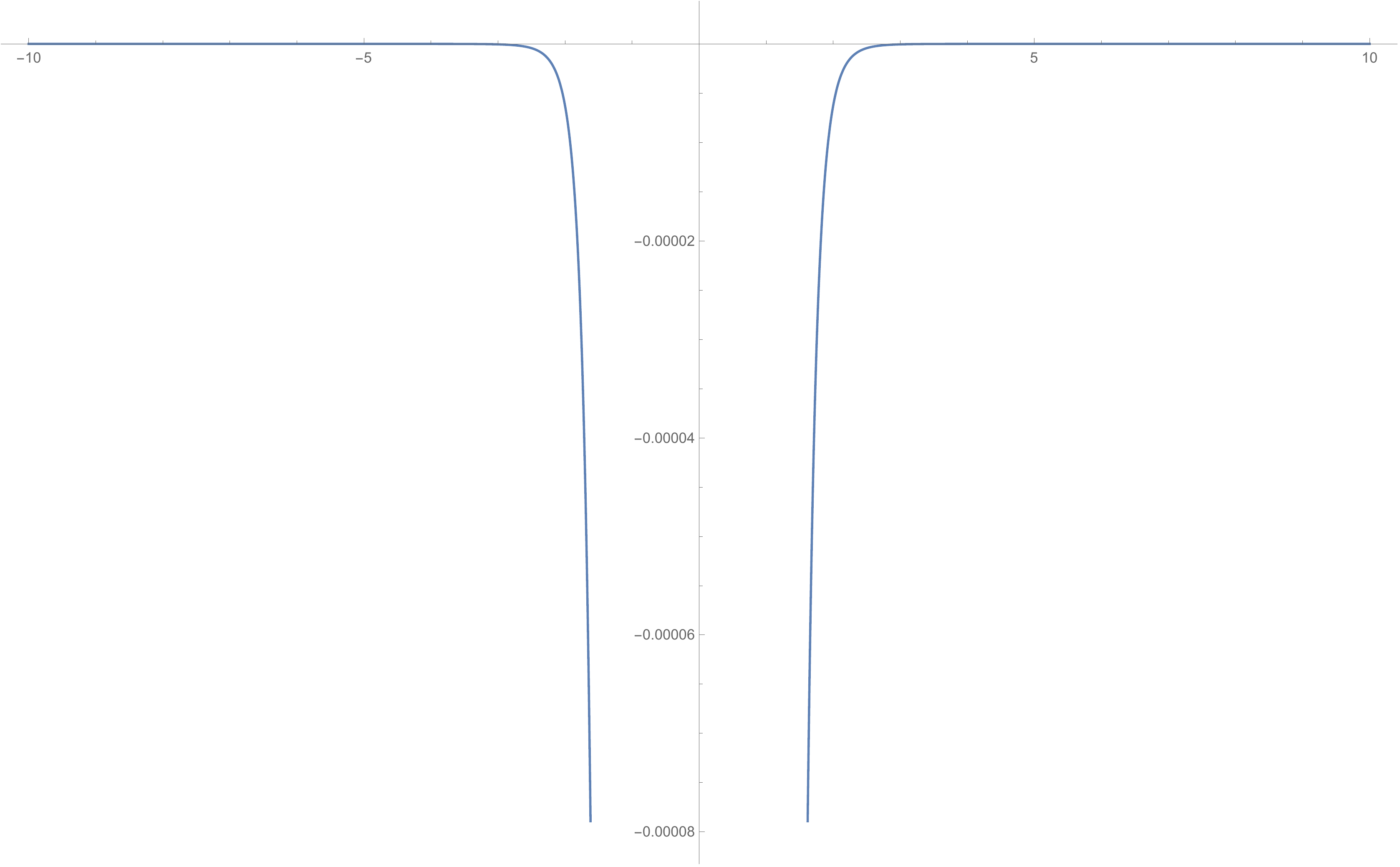}
		\caption{\centering $\Lambda _1$ for the first row}
	\end{subfigure}
	\quad\quad
	\begin{subfigure}[H]{0.25\textwidth}
	\centering
    \includegraphics[scale=0.20]{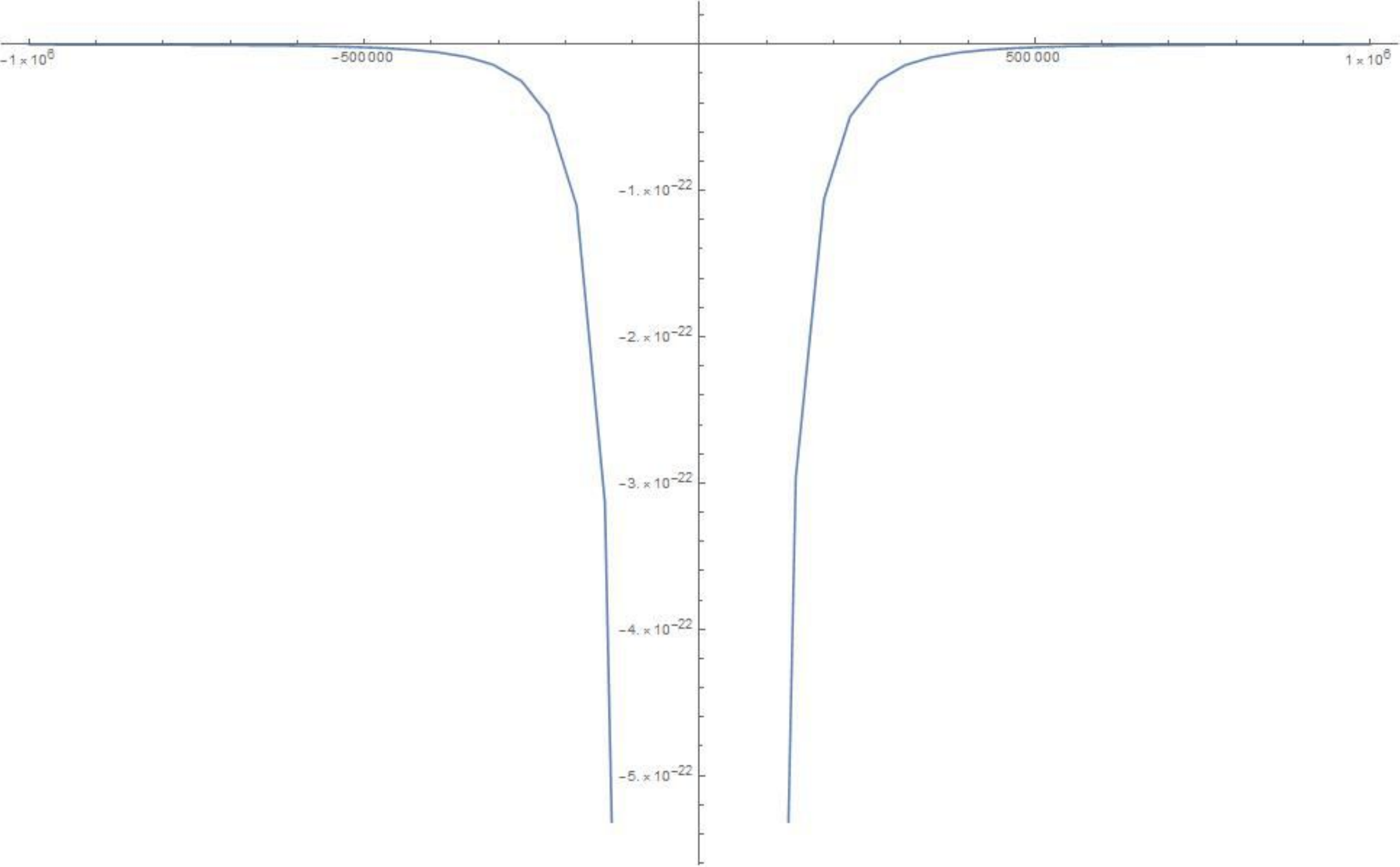}
    \caption{\centering $\Lambda _1$ for the second row}
    \end{subfigure}
    \begin{subfigure}[H]{0.25\textwidth}
	\centering
    \includegraphics[scale=0.20]{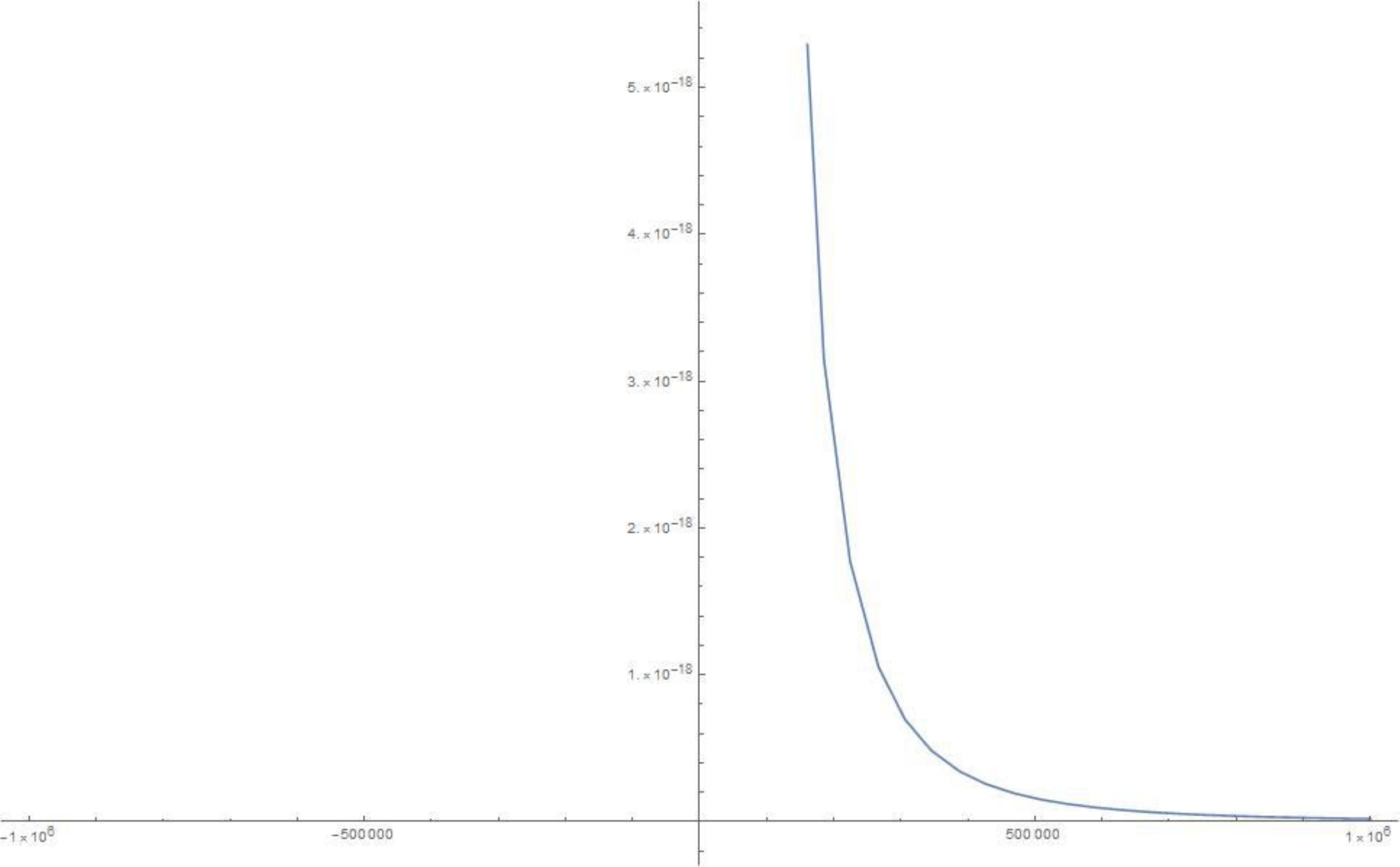}
    \caption{\centering $\Lambda _1$ for the third row}
    \end{subfigure}
    \\
	\begin{subfigure}[H]{0.25\textwidth}
	\centering
	\includegraphics[scale=0.20]{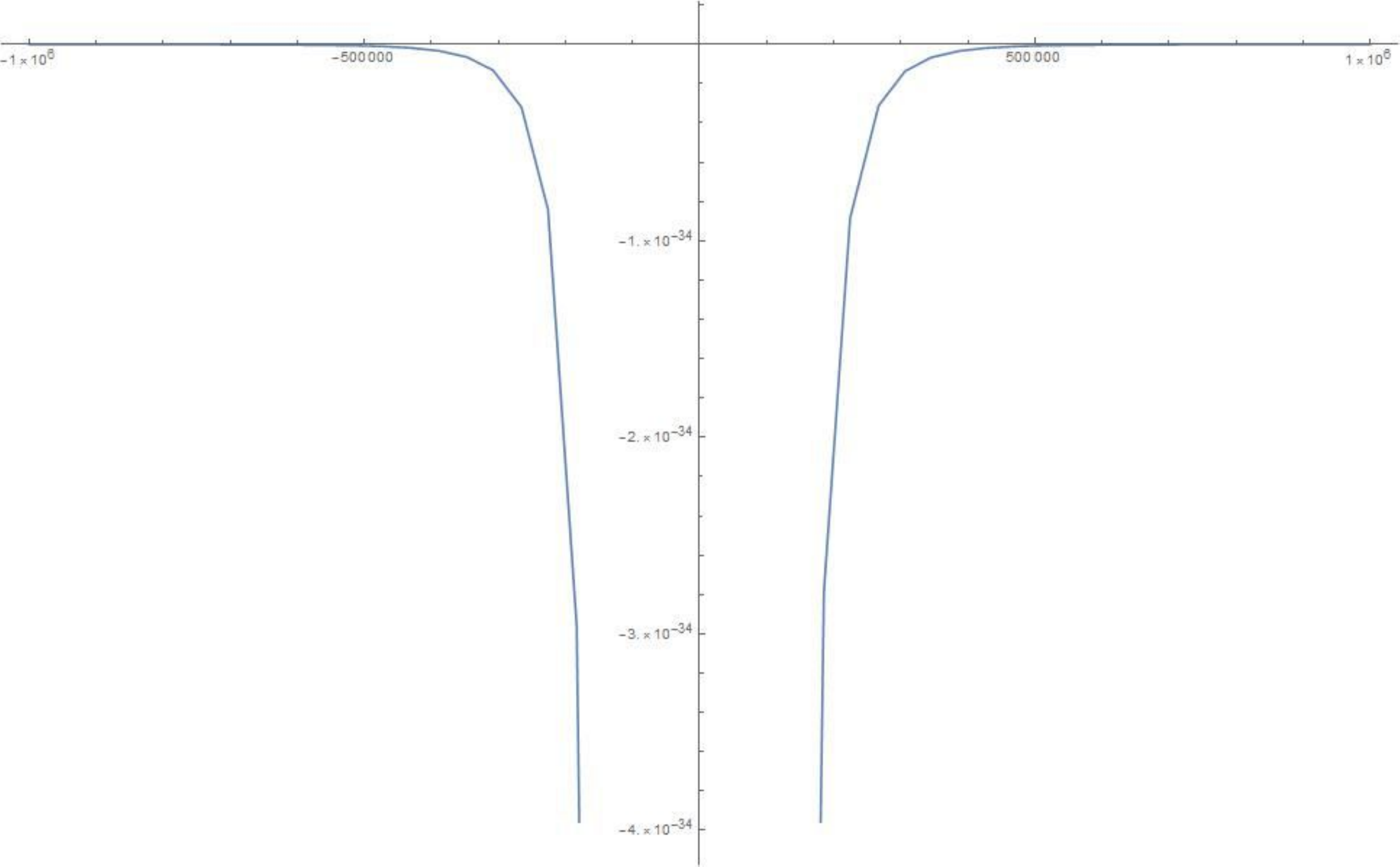}
	\caption{\centering $\Lambda _1$ for the fourth row}
	\end{subfigure}
	\quad\quad
    \begin{subfigure}[H]{0.25\textwidth}
	\centering
	\includegraphics[scale=0.20]{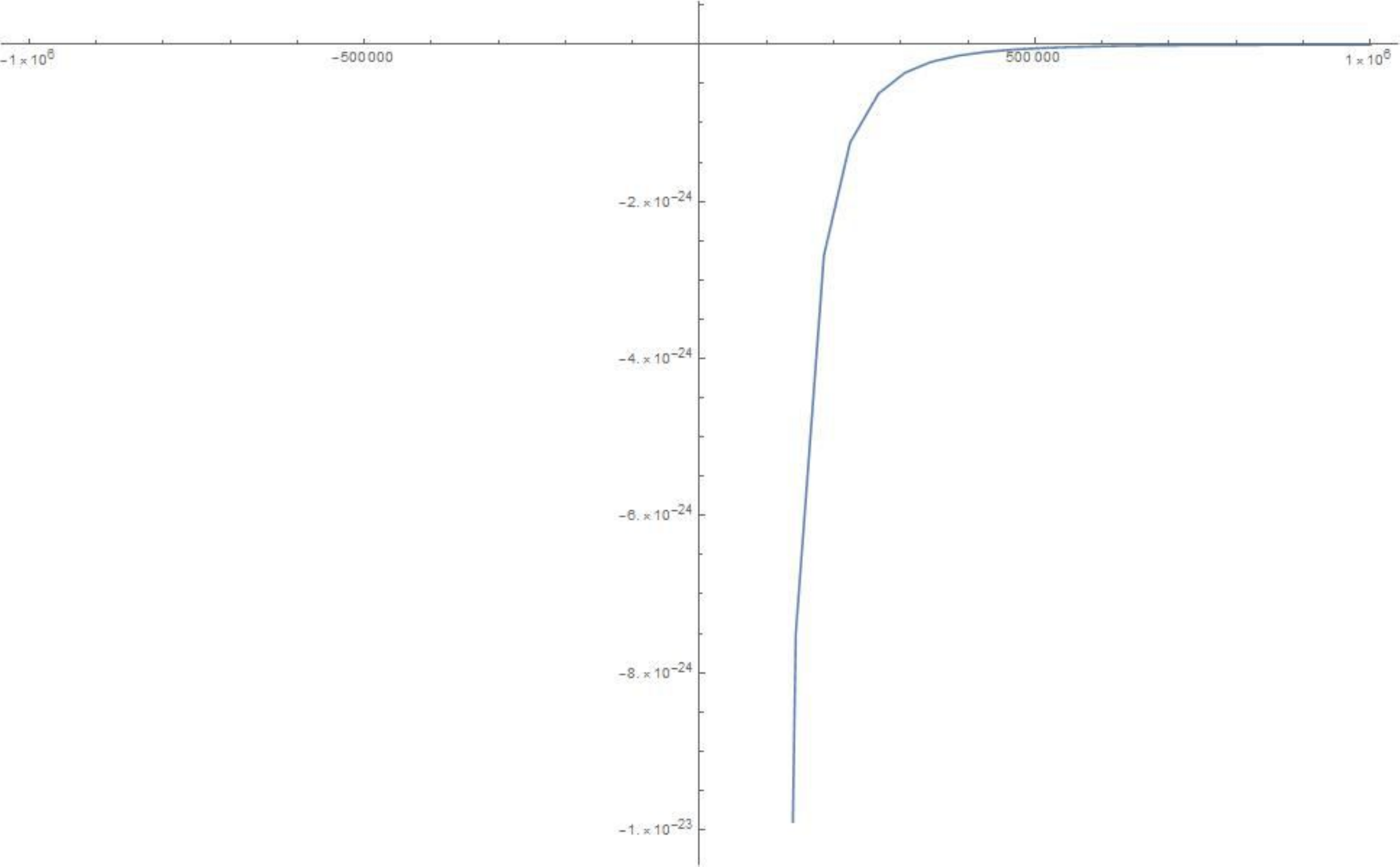}
	\caption{\centering $\Lambda _1$ for the fifth row}
	\end{subfigure}
	\begin{subfigure}[H]{0.25\textwidth}
	\centering
	\includegraphics[scale=0.20]{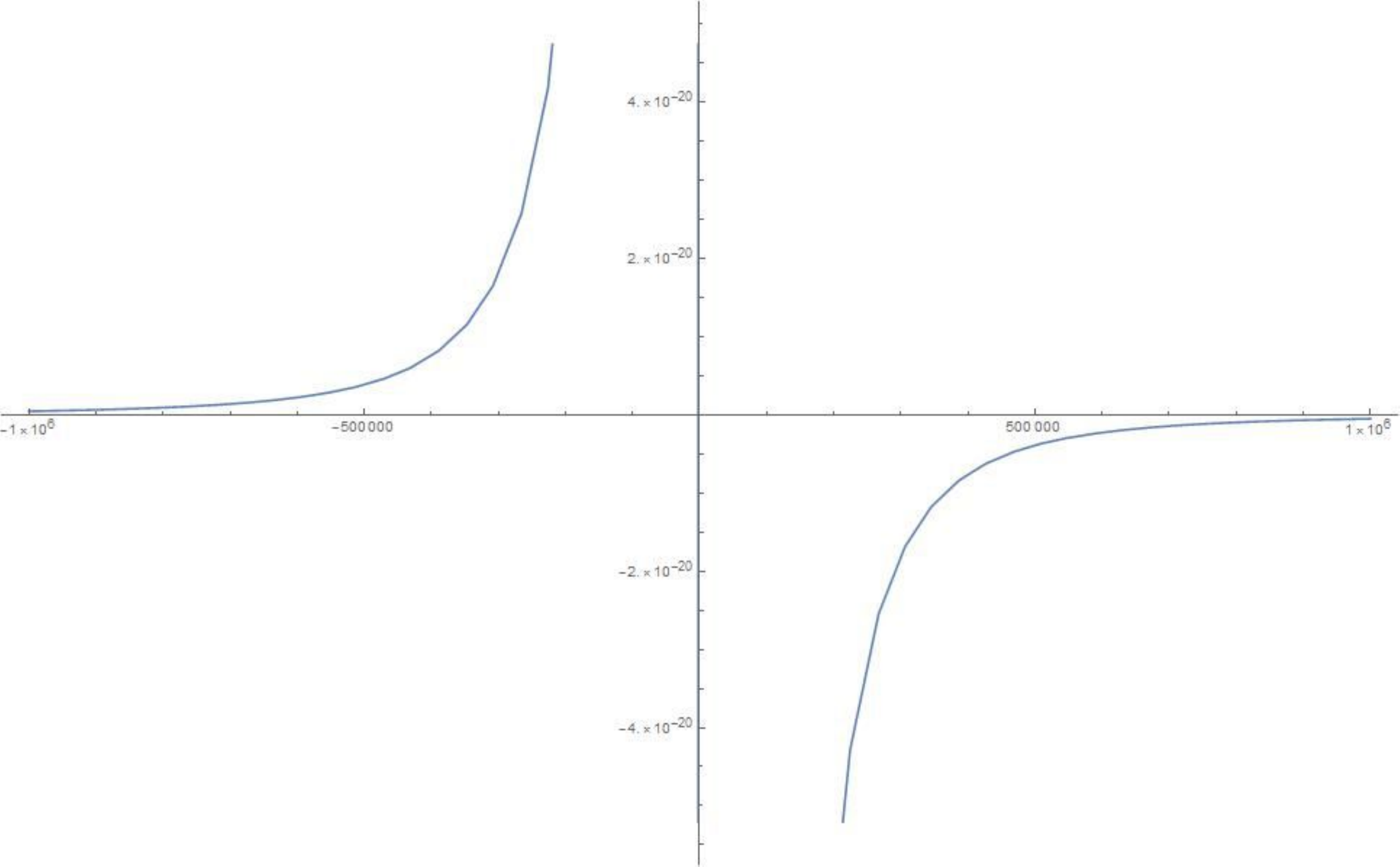}
	\caption{\centering $\Lambda _1$ for the sixth row}
	\end{subfigure}
	\\
	\begin{subfigure}[h]{0.25\textwidth}
	\centering
	\includegraphics[scale=0.20]{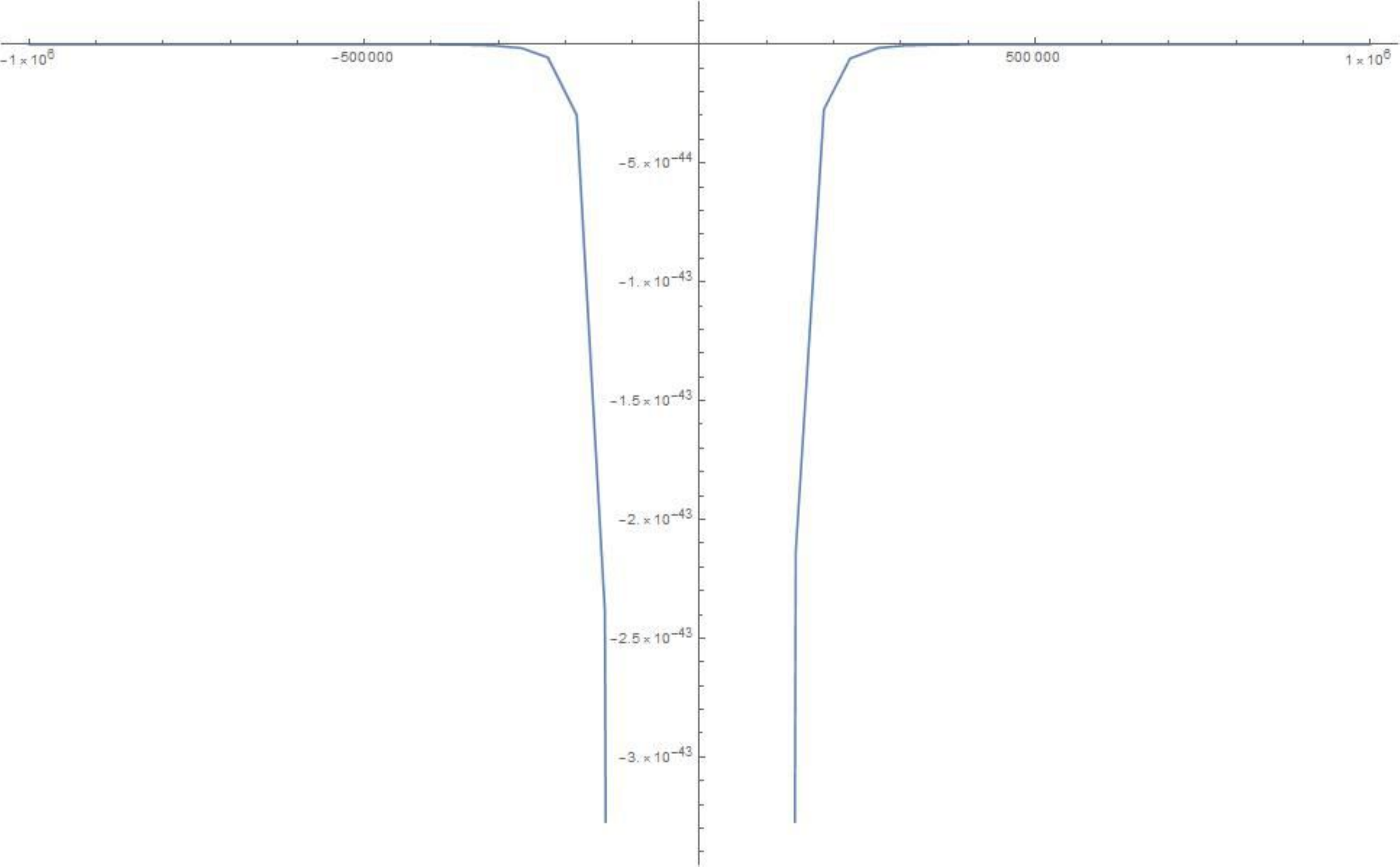}
	\caption{\centering $\Lambda _1$ for the eighth row}
    \label{draw5}		
	\end{subfigure}
	\quad\quad
	\begin{subfigure}[H]{0.25\textwidth}
	\centering
	\includegraphics[scale=0.20]{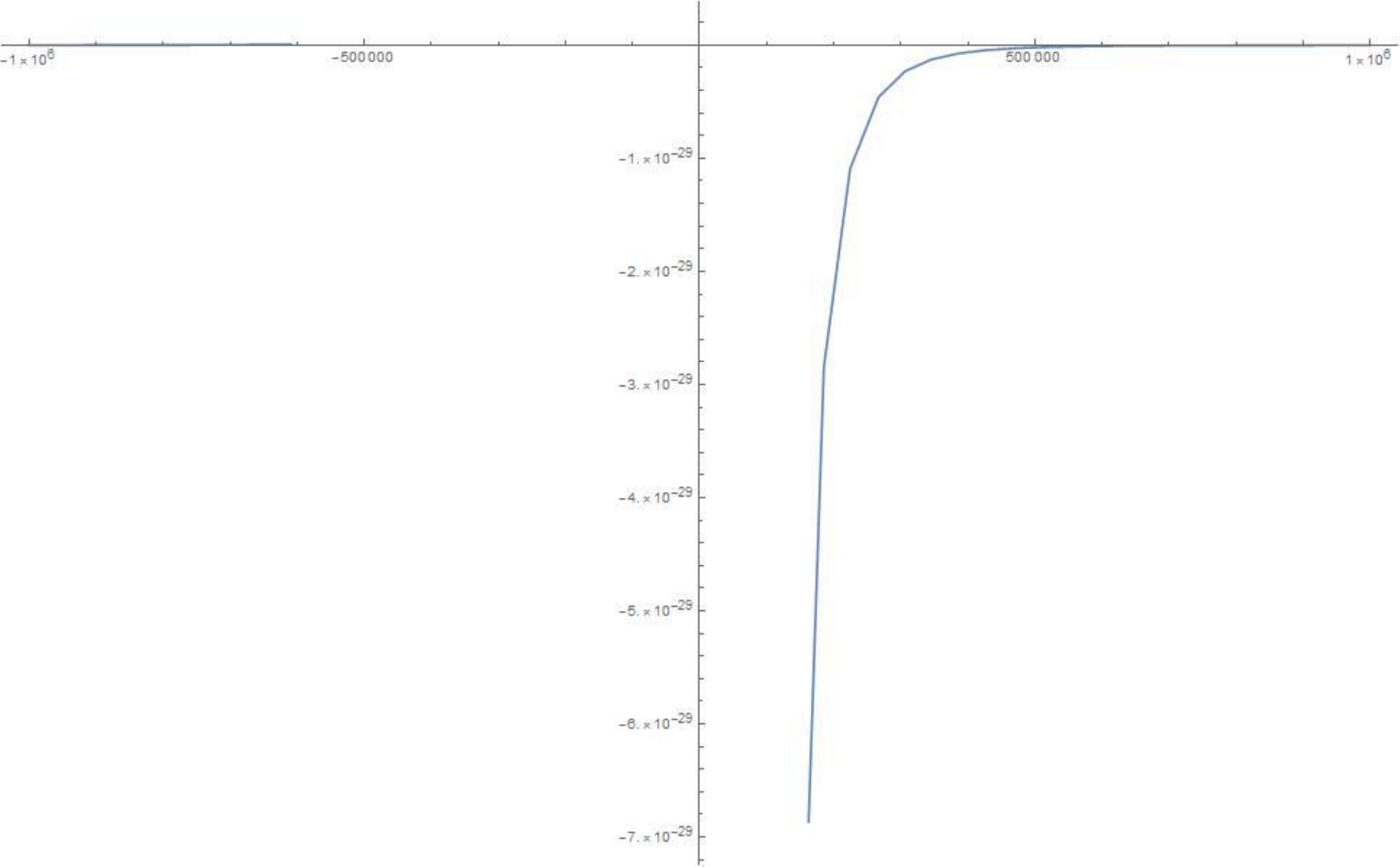}
	\caption{\centering $\Lambda _1$ for the ninth row}
	\end{subfigure}
	\begin{subfigure}[H]{0.25\textwidth}
	\centering
	\includegraphics[scale=0.20]{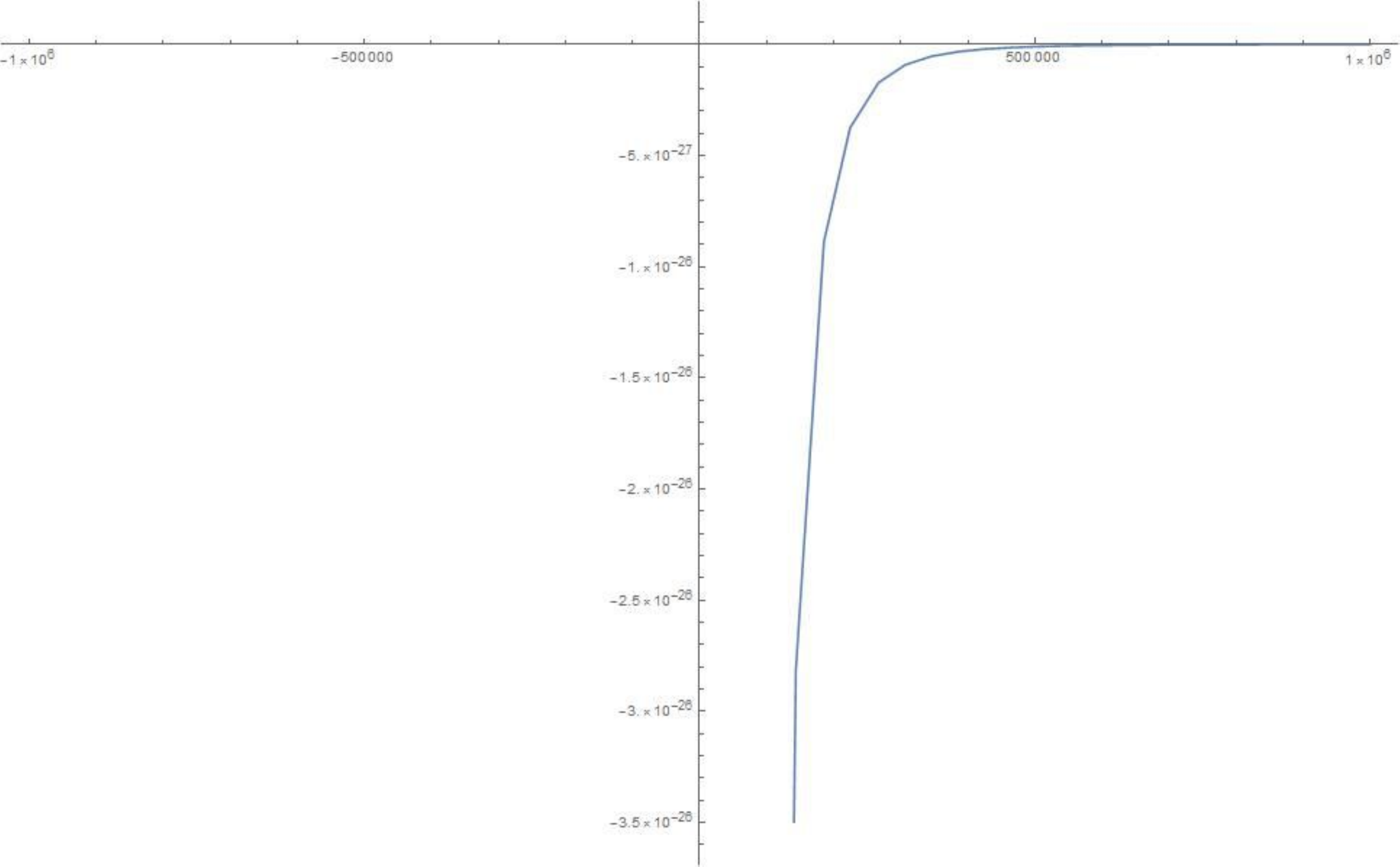}
	\caption{\centering $\Lambda _1$ for the tenth row}
	\end{subfigure}
    \caption{Graphs of some of the functions $\Lambda _1$ associated with the rows of Table \ref{table_appendix}}
\end{figure}

\section*{Acknowledgements} Y. X. Martins and L. F. A. Campos were supported by CAPES and CNPq, respectively. The authors thank the referee for the careful reading and for clarifying some mistakes.

\end{document}